\documentclass[11pt]{amsart}
\usepackage[english]{babel}
\usepackage{amssymb,amscd}

\usepackage[pagewise]{lineno}

\usepackage[all]{xy}

\usepackage{euscript}
\usepackage{amsmath}
\usepackage{ifpdf}

\usepackage{amsfonts}
\usepackage{mathrsfs}

\usepackage{amsmath}

\usepackage{mathpazo}
\usepackage{hyperref}

\newtheorem{theorem}{Theorem}[section]
\newtheorem{lemma}[theorem]{Lemma}
\newtheorem{corollary}[theorem]{Corollary}
\newtheorem{proposition}[theorem]{Proposition}

\newtheorem{remark}[theorem]{Remark}
\newtheorem{example}[theorem]{Example}

\newtheorem{definition}[theorem]{Definition}
\newtheorem{notation}[theorem]{Notation}
\newtheorem{question}[theorem]{Question}

\def\sqr#1#2{{\vcenter{\hrule height.#2pt
\hbox{\vrule width.#2pt height#1pt \kern#1pt \vrule width.#2pt}
\hrule height.#2pt}}}

\def\qed{\hspace*{\fill} $\square$}

\def\gcdimit{\textrm{\emph{G}}_C\textrm{\emph{-dim}}_R}

\def\extit{\textrm{\emph{Ext}}_{R}}

\def\gcit{\textrm{\emph{G}}_C}

\begin{document}

\title[Reduced
$\textrm{G}$-perfection and linkage relative to a semidualizing
module]{On reduced
$\textrm{G}$-perfection and horizontal linkage relative to a semidualizing
module}

\author{Cleto B.~Miranda-Neto}
\address{Departamento de Matem\'atica, Universidade Federal da
Para\'iba, 58051-900 Jo\~ao Pessoa, PB, Brazil.}
\email{cleto@mat.ufpb.br}

\author{Thyago S.~Souza}
\address{Unidade Acad\^emica de Matem\'atica, Universidade Federal de Campina Grande, 58429-970 Campina Grande, PB, Brazil.}
\email{thyago@mat.ufcg.edu.br}

\thanks{Corresponding author: C. B. Miranda-Neto  (cleto@mat.ufpb.br).}


\begin{abstract}

In their investigation of horizontal linkage of modules of finite Gorenstein dimension over a commutative, Noetherian, semiperfect (e.g., local) ring, Dibaei and Sadeghi introduced the class of reduced G-perfect modules, making use of Bass' concept of reduced grade. A few years later, the same authors extended this class by considering the relative property of reduced G$_C$-perfection, where $C$ is a semidualizing module, and studied linkage even further. In the present paper, we contribute to their theory and also generalize results of Auslander and Bridger as well as of Martsinkovsky and Strooker. Our investigation includes, for example, when reduced G$_C$-perfection is preserved by relative Auslander transpose, and how to numerically characterize horizontally linked modules under suitable conditions. Along the way, we show how to produce
reduced $\textrm{G}_C$-perfect
modules that are also $C$-$k$-torsionless (for a given integer $k\geq 0$) but fail to be $\textrm{G}_C$-perfect, and moreover we illustrate that, in contrast to the usual grade, the relative reduced grade does depend on the choice of $C$. 
\end{abstract}

\maketitle

\section{Introduction}

Since the seminal work of Auslander and Bridger \cite{AB}, the theory of Gorenstein dimension -- G-dimension, for short -- has been further expanded and received increasing attention. The present paper is concerned with the interplay between  
G-dimension taken with respect to a semidualizing module $C$ and the concept of reduced grade relative to $C$ (the latter extending the notion of reduced grade due to Bass \cite{Bass}), over a ring $R$ belonging to a suitable class which contains the class of (commutative, Noetherian) local rings.

More specifically, as anticipated in the abstract, one of our main goals is to contribute to the theory mostly by extending, non-trivially, some of the results of
Dibaei and Sadeghi \cite{link2013} to the relative context taking into account the presence of $C$. Such relative theory, including the property of reduced G$_C$-perfection as well as the relative operations of Auslander transpose and horizontal linkage, was developed by the same authors in \cite{link2015}, \cite{link2017}, also Sadeghi \cite{link2016}; here it should be mentioned that the idea of relative Auslander transpose is due to Foxby \cite{foxby72}. Along the way, we furthermore improve results of Auslander and Bridger \cite{AB} as well as of Martsinkovsky and Strooker \cite{link2004}, and provide examples illustrating properties which, as far as we know, have gone unnoticed in the literature.

It is worth mentioning that the notion of semidualizing module has become a central object
in homological algebra since its inception independently by Foxby \cite{foxby72}, Golod \cite{golod84}, and Vasconcelos
\cite{vasconcelos74}. Many authors have studied and justified the concept under different viewpoints and objectives; we refer, besides the references already mentioned and just to quote a few, to \cite{chris2001},  
\cite{holmjorgensen2006}, \cite{salimisather2012},
\cite{salimiyassemi2012}, \cite{Yassemi2014},
\cite{salimi2015}, \cite{satherwhite2010i},
\cite{satherwhite2010ii}, \cite{satheryassemi2009}, \cite{TakWhi2010},
\cite{white2010}.

Next we briefly describe the content of the paper. 

Basic notions and auxiliary facts are collected in Section \ref{aux}. The central concept of reduced $\textrm{G}_C$-perfect module is defined and studied systematically in Section \ref{sec31}, where in particular an equality is given connecting the $\textrm{G}_C$-dimension of such a module with the
$\textrm{G}_C$-dimensions of certain related modules, for instance its relative Auslander transpose (see Theorem \ref{teo.Gperf.redu.formula.da.GCdim.generalizado}). The same section provides examples, which seem to be new, illustrating that (in  grade zero) the relative reduced grade may depend on the choice of $C$ (see Example \ref{conductor}). In Section \ref{sec32}, some results of Auslander and Bridger \cite{AB} are extended to the context of
$\textrm{G}_C$-dimension, and formulas are presented relating the
grade and the reduced grade with respect to $C$. In Example \ref{ex2.GCperfeito.reduzido} a recipe is given to produce reduced $\textrm{G}_C$-perfect
modules that are not $\textrm{G}_C$-perfect.
The main result of the section is Theorem \ref{teo.inspirado.no.cor.3.16.paper3} as it investigates conditions under which
the reduced $\textrm{G}_C$-perfect property is preserved under the operation of taking
Auslander transpose with respect to $C$. A main consequence is Corollary \ref{cor.3.16.paper3.generalizado}, which is a generalization  of Dibaei and Sadeghi \cite[Corollary 3.6]{link2013}. In Section \ref{sec33}, the interplay between reduced $\textrm{G}_C$-perfection and horizontal linkage is treated, and improvements are given of some
results of Martsinkovsky and Strooker \cite{link2004} and of
Dibaei and Sadeghi \cite{link2013}; for instance, Theorem \ref{prop.teo1.paper1.generalizado} generalizes \cite[Theorem 1]{link2004}, and Theorem \ref{teo.prop.3.5ii.paper3.generalizada} significantly extends \cite[Proposition 3.5(ii)]{link2013}; the latter theorem, in an appropriate setting, characterizes horizontally linked modules numerically. Finally, in Section \ref{sec34}, the class of $C$-$k$-torsionless modules, for an integer $k\geq 0$, is invoked and shown to afford no general containment with respect to the class of reduced
$\textrm{G}_C$-perfect modules, as observed in Remark \ref{ex0.GCperfeito.reduzido.e.cktorsionless}. A necessary and sufficient
condition for a reduced $\textrm{G}_C$-perfect module to be
$C$-$k$-torsionless is provided by
Proposition \ref{prop.cktorsionless.gcperfeito}, and in Example \ref{ex.GCperfeito.reduzido.e.cktorsionless} a family of modules lying in the intersection of the two classes is described.

\section{Preliminaries and auxiliary results}\label{aux}

Throughout the paper, by {\it ring} we mean a semiperfect Noetherian commutative
ring with multiplicative identity $1$.  Wherever it appears, $R$ stands for a ring. By a {\it finite} $R$-module we mean a finitely generated
$R$-module. For general homological algebra, we refer to \cite{W}.

Recall that, in the present commutative context, a ring is semiperfect if and only if it is a direct product of finitely many local rings (see \cite[Theorem 23.11]{Lam}). Thus, every local commutative ring is semiperfect. We point out that semiperfection is (implicitly) required in this paper essentially due to the fact that, over any ring with this property, every finite module possesses a projective cover, hence a minimal presentation.
Moreover, there is a useful operator which is well-defined in this setting and allows for a deeper investigation of horizontal linkage of modules; such operator, typically denoted by $\lambda$ (see \cite{link2004}), will be considered later on.

We invoke some well-known important notions and auxiliary facts.

\begin{definition}\label{def.semidualizante}\rm
A finite $R$-module $C$ is called a \emph{semidualizing} module if the following properties hold:
\begin{itemize}
\item[(i)] The homothety morphism $R \rightarrow \textrm{Hom}_{R}(C, C)$ is an isomorphism;
\item[(ii)] $\textrm{Ext}_{R}^i(C, C) = 0$ for all $i > 0$.
\end{itemize}
If, in addition, $C$ has finite injective dimension over $R$, then $C$ is said
to be a \emph{dualizing} module.
\end{definition}

\begin{remark} \label{ex.semidualizante}\rm It is evident that $R$ is semidualizing as a module over itself. If $R$ is a Cohen-Macaulay local ring having a canonical module $\omega_R$ (which exists if and only if $R$ is a homomorphic image of a Gorenstein local ring), then $\omega_R$ is a
dualizing module. There are plenty of examples of local rings admitting a semidualizing module which is neither free nor dualizing; see, for instance, \cite[Example 5.3]{arayaiima} and \cite[Example 1.2]{JLS}.
\end{remark}

\begin{notation}\rm
In the entire paper, $C$ stands (wherever it appears) for a semidualizing module over a given ring $R$.
We denote by $(-)^C = \textrm{Hom}_{R}(-, C)$  the duality
functor with respect to $C$. In the particular case
where $C = R$, we use the typical notation $(-)^* = \textrm{Hom}_{R}(-, R)$. If
$M$ is an $R$-module, we denote by $\sigma_M^C \colon M \rightarrow M^{C C}$
the natural biduality map with respect to $C$, which is defined by
$\sigma_M^C(x)(f) = f(x)$ for any $x \in M$ and $f \in M^C$.
\end{notation}

\begin{definition}\rm A finite $R$-module $M$ is called \emph{$C$-reflexive} if $\sigma_M^C$ is an isomorphism.
\end{definition}

\begin{definition}\label{def.Creflexivo}\rm
A finite $R$-module $M$ is said to be \emph{totally $C$-reflexive} if $M$ is $C$-reflexive and
$$\textrm{Ext}_{R}^i(M,\,C) \, = \, \textrm{Ext}_{R}^i(M^{C},\,C) \, = \, 0$$ for all $i > 0$.
In particular, if $C = R$, a totally $R$-reflexive module
is simply a totally reflexive module in the usual sense.
\end{definition}

The result below describes the behavior of total
$C$-reflexivity along a short exact sequence.

\begin{proposition} \label{prop.item2.golod}
$($\cite[Proposition 5.1.1 and Proposition 5.1.3]{sather}$)$. Consider a short
exact sequence of finite $R$-modules $0 \rightarrow M' \rightarrow M
\rightarrow M'' \rightarrow 0$. If $M', M''$ or $M, M''$ are
totally $C$-reflexive, then the third module is totally
$C$-reflexive as well. If $M', M$ are totally $C$-reflexive and
$\textrm{\emph{Ext}}_{R}^1(M'', C) = 0$, then $M''$ is totally
$C$-reflexive.
\end{proposition}

Next, we recall the definition of ${\rm G}_C$-dimension, i.e. Gorenstein dimension with respect to $C$. For a comprehensive treatment of the general theory on this invariant, we refer to \cite{sather}.

\begin{definition}\rm
Let $M$ be a finite $R$-module. A $\textrm{G}_C$\emph{-resolution} of $M$ is an exact sequence
$$\cdots \rightarrow X_i \rightarrow X_{i-1} \rightarrow \cdots
\rightarrow X_0 \rightarrow M \rightarrow 0$$ such that $X_i$ is
totally $C$-reflexive for all $i \geq 0$. Note that every finite
projective $R$-module is totally $C$-reflexive (see
\cite[Proposition 2.1.13(b)]{sather}), and hence every finite $R$-module admits
a $\textrm{G}_C$-resolution. The smallest non-negative integer $n$
for which there exists a $\textrm{G}_C$-resolution
$$0 \rightarrow X_n \rightarrow X_{n-1} \rightarrow \cdots
\rightarrow X_0 \rightarrow M \rightarrow 0$$ is the
$\textrm{G}_C$\emph{-dimension} of $M$, and in this case we write $\textrm{G}_C\textrm{-dim}_R(M)=n$. If such $n$ does not exist, we say that $M$ has {\it infinite} $\textrm{G}_C$-{\it dimension} and we write $\textrm{G}_C\textrm{-dim}_R(M)=\infty$. Note that by taking $C = R$ we recover the usual Gorenstein dimension $\textrm{G}\textrm{-dim}_R(M)$ of $M$.
\end{definition}

\begin{remark} \label{obs.carac.GCdim.zero}\rm A finite $R$-module $M$ is totally
$C$-reflexive if and only if $\textrm{G}_C\textrm{-dim}_R(M) =
0$. If $M$, $N$ are finite $R$-modules, \cite[Proposition 2.1.4]{sather} gives that
$\textrm{G}_C\textrm{-dim}_R(M\oplus N) = 0$ if and only if
$\textrm{G}_C\textrm{-dim}_R(M) = 0$ and
$\textrm{G}_C\textrm{-dim}_R(N) = 0$.
\end{remark}

\begin{proposition} \label{prop.6.1.7.sather}
$($\cite[Proposition 6.1.7]{sather}$)$. Let $M$ be a non-zero finite $R$-module
of finite $\gcit$-dimension. Then
$$\gcdimit (M) \, = \,  \sup \{ i \geq 0 \mid \extit^i(M,\,C) \neq 0 \}.$$
\end{proposition}

\begin{proposition} \label{prop.item6.golod}
$($\cite[Proposition 6.1.8]{sather}$)$. Consider a short exact sequence of finite
$R$-modules $0 \rightarrow M' \rightarrow M \rightarrow M''
\rightarrow 0$. If two of them have finite $\gcit$-dimension,
then so does the third.
\end{proposition}

\begin{proposition} \label{prop.lema.1.9.gerko}
$($\cite[Proposition 6.1.10]{sather}$)$. Consider a short exact sequence of finite
$R$-modules of finite $\gcit$-dimension $0 \rightarrow M'
\rightarrow M \rightarrow M'' \rightarrow 0$. If $\gcdimit
(M)=0$ and $\gcdimit (M'')>0$, then $\gcdimit (M') =
\gcdimit (M'')-1$.
\end{proposition}

The following theorem is the extension of the classical Auslander-Bridger formula
(see \cite[Theorem 4.13(b)]{AB}) to the $\textrm{G}_C$-dimension
setting. It is worth recalling that if the projective dimension of a module is finite then so is the $\textrm{G}_C$-dimension
and these numbers coincide (see \cite[Corollary 6.4.5]{sather}).

\begin{theorem}\label{teo.propriedades.Gcdimensao}
$($\cite[Theorem 4.4]{geng2013}$)$. Let $R$ be a local ring and let $M$ be a
non-zero finite $R$-module of finite $\textrm{\emph{G}}_C$-dimension.
Then
$$\gcdimit (M) \, = \, \textrm{\emph{depth}}(R) -
\textrm{\emph{depth}}_R(M).$$
\end{theorem}

\begin{proposition} \label{prop.prop.1.3.gerko}
$($\cite[Proposition 1.3]{gerko2001}$)$. Let $R$ be a local ring. Then
$C$ is a dualizing module if and only if $\gcdimit (M)< \infty$
for every finite $R$-module $M$.
\end{proposition}

\begin{proposition} \label{prop.obs6.1.9.sather}
$($\emph{\cite[Remark 6.1.9]{sather}}$)$. Let $0 \rightarrow M'
\rightarrow M \rightarrow M'' \rightarrow 0$ be a short exact sequence
of finite $R$-modules. Then we have the following inequalities:
\begin{itemize}
             \item[(i)] $\gcdimit (M') \leq \sup \{\gcdimit (M), \gcdimit (M'') \};$
             \item[(ii)] $\gcdimit (M) \leq \sup \{\gcdimit (M'), \gcdimit (M'') \}.$
\end{itemize}
\end{proposition}

We now consider the following general concept of Auslander transpose of a
finite module relative to a semidualizing module (see \cite{foxby72}, also
\cite{geng2013}).

\begin{definition}\label{def.TrC}\rm
Let $P_1 \xrightarrow{f} P_0 \rightarrow M \rightarrow 0$ be a
projective presentation of a finite $R$-module $M$. Then, applying
$\textrm{Hom}_{R}(-, C)$, we get the exact sequence
$$0 \rightarrow M^{C} \rightarrow P_0^{C} \xrightarrow{f^{C}}
P_1^{C} \rightarrow \textrm{Coker} (f^{C}) \rightarrow 0.$$ We
call $\textrm{Coker}(f^{C})$ a \emph{transpose of $M$ with
respect to $C$}, and denote it by $\textrm{Tr}_C(M)$. In the case where $C = R$, this notion coincides with the usual
\emph{Auslander transpose}, denoted by $\textrm{Tr}(M)$.
\end{definition}

Let $\textrm{add}_R(C)$ denote the category consisting of
all finite $R$-modules isomorphic to direct summands of finite direct
sums of copies of $C$. Two finite $R$-modules $M$ and $N$ are said to be
$\textrm{add}_R(C)$\emph{-equivalent}, which is denoted by $M \approx_C
N$, if there exist $X, Y$ in  $\textrm{add}_R(C)$ such that
$M \oplus X \cong N \oplus Y$.

\begin{remark}\label{obs.TrC}\rm
Let $M$ be a finite $R$-module. It is clear that the module $\textrm{Tr}_C(M)$ depends on the choice of the projective presentation of $M$, but it
is unique up to $\textrm{add}_R(C)$-equivalence (note that minimal projective presentations, which exist because $R$ is semiperfect, yield isomorphic transposes). As a consequence, $\textrm{Ext}_{R}^i(\textrm{Tr}_C(M), C)$ is unique up to isomorphism for any $i> 0$ (see \cite[Remark 2.4(1)]{geng2013}). It is also easy to see that $\textrm{G}_C\textrm{-dim}_R(\textrm{Tr}_C(M))$ is well-defined.
\end{remark}

\begin{proposition} \label{prop.lemma.2.12.paper5}
$($\emph{\cite[Lemma 2.12]{link2016}}$)$. For a finite $R$-module $M$, there
exists a short exact sequence
$$0 \rightarrow M \rightarrow
\textrm{\emph{Tr}}_C(\textrm{\emph{Tr}}_C(M)) \rightarrow X
\rightarrow 0$$ where $\textrm{\emph{G}}_C\textrm{\emph{-dim}}_R(X) = 0$.
\end{proposition}

\begin{proposition} \label{prop.prop.2.6.geng}
$($\cite[Lemma 2.5 and Proposition 2.7]{geng2013}$)$. Let $M$ be a finite
$R$-module.
\begin{itemize}
             \item[(i)] There is an exact sequence
             $$0 \rightarrow \textrm{\emph{Ext}}_{R}^1(\textrm{\emph{Tr}}_C(M),\,C) \rightarrow M \xrightarrow{\sigma_M^C}M^{CC}
             \rightarrow \textrm{\emph{Ext}}_{R}^2(\textrm{\emph{Tr}}_C(M),\,C) \rightarrow 0;$$
             \item[(ii)] If $\textrm{\emph{G}}_C\textrm{\emph{-dim}}_R(M) = 0$
             then $\textrm{\emph{G}}_C\textrm{\emph{-dim}}_R(M^C) = 0$;
             \item[(iii)] $\textrm{\emph{G}}_C\textrm{\emph{-dim}}_R(M) = 0$ if and only if $\textrm{\emph{G}}_C\textrm{\emph{-dim}}_R(\textrm{\emph{Tr}}_C(M)) = 0$.
\end{itemize}
\end{proposition}

The proof of the following proposition is similar to the proof of
\cite[Remark 2.1(i)]{link2015}.

\begin{proposition} \label{lema.tr.e.trC}
Let $M$ be a finite $R$-module. Then, $\textrm{\emph{Tr}}_C(M)
\approx_C \textrm{\emph{Tr}}(M)\otimes_R C$. Moreover, if
$\textrm{\emph{Tr}}(M)$ and $\textrm{\emph{Tr}}_C(M)$ come from
the same projective presentation of $M$, then
$\textrm{\emph{Tr}}_C(M) \cong \textrm{\emph{Tr}}(M) \otimes_R C$.
\end{proposition}

If $\alpha \colon P \rightarrow M$ is an epimorphism, with $P$
a projective $R$-module, then the kernel of $\alpha$ is denoted $\Omega M$ and dubbed \emph{syzygy module} of $M$. By Schanuel's lemma, the projective equivalence class of $\Omega M$ is uniquely defined. Moreover, as $R$ is semiperfect, we may take the epimorphism $\alpha$ to be minimal (i.e., $P$ is a projective cover of $M$), hence $\Omega M$ is uniquely determined up to isomorphism. Now, as in \cite{link2004}, we consider the operator $$\lambda \, =  \, \Omega \textrm{Tr}$$ which, for any finite $R$-module $M$, gives the composite $\Omega \textrm{Tr}(M)$ defined from a minimal projective presentation $P_1\rightarrow P_0\rightarrow M\rightarrow  0$. This operator was, in fact, first introduced in \cite{AB}, with different notation.

\begin{definition}\label{def.LH} \rm
$($\cite[Definition 3]{link2004}$)$. Two finite $R$-modules $M$, $N$ are said
to be \emph{horizontally linked} if $M \cong \lambda N$ and $N
\cong \lambda M$. In case $M$ and $\lambda M$ are horizontally linked (i.e. $M \cong \lambda^2 M$), we simply say that the module $M$ is horizontally linked.
\end{definition}

Also recall that a finite module is called \emph{stable} if it has no non-zero projective direct summand. Here is a characterization of horizontally linked
modules in terms of stability and the kernel of the evaluation map $\sigma_M$.

\begin{theorem}\label{teo2.paper1}
$($\cite[Theorem 2]{link2004}$)$. A finite $R$-module $M$ is horizontally
linked if and only if $M$ is stable and ${\rm Ext}_{R}^1(\textrm{\emph{Tr}}(M), R) = 0$.
\end{theorem}

Recall that the {\it grade} of a finite $R$-module $M$ is defined as $\textrm{grade}_R(M) = {\rm inf}\{i \geq 0 \mid \textrm{Ext}_{R}^i(M, R) \neq 0\}$ (for the trivial module we set $\textrm{grade}_R(0)=\infty$). This is an important basic invariant and gives a lower bound to the $\textrm{G}_C$-dimension.

\begin{lemma} \label{lema.grade.e.gradeC}
Let $M$ be a non-zero finite $R$-module. Then
$${\rm grade}_R(M) \, = \, {\rm inf} \{i \geq 0 \mid
{\rm Ext}_{R}^i(M,\,C) \neq 0\}.$$ In particular,
$\textrm{\emph{grade}}_R(M) \leq
\textrm{\emph{G}}_C\textrm{\emph{-dim}}_R(M)$.
\end{lemma}
\begin{proof}
By \cite[Proposition 2.1.16(c)]{sather}, we have $\textrm{ann}_R(M)C \neq C$. Now, by \cite[Corollary 3.2]{geng2013} (see also \cite[Theorem 2.2.6(a)]{sather}), a sequence ${\bf x} \subset R$ is an $R$-sequence if and only if ${\bf x}$ is a $C$-sequence. Hence, using moreover \cite[Proposition 1.2.10(e)]{CMr}, we get $\textrm{grade}_R(M)  =  \textrm{grade}_R(\textrm{ann}_R(M), R) = \textrm{grade}_R(\textrm{ann}_R(M), C) = \inf \{i \geq 0 \mid \textrm{Ext}_{R}^i(M, C)\neq 0\}$. The last assertion then follows from Proposition \ref{prop.6.1.7.sather}. \qed
\end{proof}

\medskip

A finite $R$-module $M$ is called a
\emph{$C$-syzygy module} if there exists an exact sequence $0 \rightarrow M \rightarrow P \otimes_R C$, where $P$
is a finite projective $R$-module.

\begin{lemma}\label{obs.syzygy}\rm
Let $M$ be a finite $R$-module.
\begin{itemize}
  \item[(i)] If $M$ is an $R$-syzygy, then $M$ is a $C$-syzygy;

  \item[(ii)] $M$ is a $C$-syzygy if and only if
$\textrm{Ext}_{R}^1(\textrm{Tr}_C(M), C) = 0$.
\end{itemize}
\end{lemma}
\begin{proof} (i) Suppose that $M$ embeds into a finite projective $R$-module $P\cong P^{CC}$. Now, choose an exact sequence $F\rightarrow P^C\rightarrow 0$, where $F$ is a finite free $R$-module. Applying $\textrm{Hom}_{R}(-, C)$, we obtain that $P^{CC}$ injects into $F^C\cong F\otimes_RC$ and hence so does $M$.

\medskip

\noindent (ii) Assume first that $\textrm{Ext}_{R}^1(\textrm{Tr}_C(M), C) = 0$. By Proposition \ref{prop.prop.2.6.geng}, $M$ injects into $M^{CC}$ via $\sigma^C_M$. Choose a surjection $F\rightarrow M^C\rightarrow 0$, where $F$ is a finite free $R$-module. Applying $\textrm{Hom}_{R}(-, C)$, we get that $M^{CC}$ injects into $F^C\cong F\otimes_RC$, hence so does $M$. Conversely, suppose that there is an injection $\iota \colon M\rightarrow P \otimes_R C$ for some finite projective $R$-module $P$. It induces a map ${\iota}^{CC} \colon M^{CC}\rightarrow (P\otimes_RC)^{CC}$. The module $P\otimes_RC$ lies in $\textrm{add}_R(C)$ by \cite[Theorem 3.1(1)]{geng-ding}, i.e. it is a direct summand of a finite direct sum of copies of $C$, which must have G$_C$-dimension zero. By Remark \ref{obs.carac.GCdim.zero}, we get $\textrm{G}_C\textrm{-dim}_R(P\otimes_RC) = 0$. In particular, the map $\sigma^C_{P\otimes_RC}$ is an isomorphism. By the natural commutative diagram
\newcommand{\f}{\operatorname{f}}
\begin{equation*}
\begin{CD}
M @>{\iota}>>P\otimes_RC\\
@VV{\sigma_M^C}V    @VV{\sigma_{P\otimes_RC}^C}V \\
M^{CC} @>{{\iota}^{CC}}>>(P\otimes_RC)^{CC}
\end{CD}
\end{equation*}
it follows that $\sigma^C_{M}$ is injective, as needed. \qed
\end{proof}

\section{Reduced $\textrm{G}_C$-perfect modules}\label{sec31}

In this section, we study reduced
$\textrm{G}_C$-perfection systematically. One of our main results provides a formula that relates
the $\textrm{G}_C$-dimension of a reduced $\textrm{G}_C$-perfect
$R$-module $M$ with the $\textrm{G}_C$-dimensions of $\textrm{Tr}_C(M)$ and $\textrm{Ext}_{R}^{\textrm{G}_C\textrm{-dim}_R(M)}(M,C)$. We conclude the section with some consequences involving
the operator $\lambda = \Omega \textrm{Tr}$. First, we need to provide some central notions.

As introduced in \cite[$\S$8]{Bass}, the {\it reduced grade} of a finite $R$-module $M$ is
$$\textrm{r.grade}_R(M) \, = \, \inf \{i > 0 \mid \textrm{Ext}_{R}^i(M,\,R) \neq
0\}.$$ The following relative notion, suggested in \cite{link2015}, extends this concept by taking into account a semidualizing module.

\begin{definition}\rm
The \emph{reduced grade} of a finite $R$-module $M$ with respect to $C$ is defined as
$$\textrm{r.grade}_R(M, C) \, = \, \inf \{i > 0 \mid \textrm{Ext}_{R}^i(M,\,C) \neq
0\}.$$
\end{definition}

\begin{remark}\label{obs.rgrade.com.respeito.C}\rm
Let $M, N$ be finite $R$-modules. The following properties hold:
\begin{itemize}
  \item[(i)] $\textrm{grade}_R(M) \leq \textrm{r.grade}_R(M, C)$, with equality if $\textrm{grade}_R(M) >0$.
  \item[(ii)] $\textrm{Ext}_{R}^i(M, C) = 0$ for all $i >0$ if and only if $\textrm{r.grade}_R(M, C) = \infty$.
  \item[(iii)] If $M$ has positive $\textrm{G}_C$-dimension, then $\textrm{r.grade}_R(M, C)\leq \textrm{G}_C\textrm{-dim}_R(M)$. Thus we have a chain of inequalities $$\textrm{grade}_R(M) \, \leq \, \textrm{r.grade}_R(M, C) \, \leq \, \textrm{G}_C\textrm{-dim}_R(M).$$
  \item[(iv)] If $M \approx_C N$, then $\textrm{r.grade}_R(M, C) = \textrm{r.grade}_R(N, C)$. In particular, the reduced grade of $\textrm{Tr}_C(M)$ with respect to $C$ is well-defined.
\end{itemize}
\end{remark}

By Remark \ref{obs.rgrade.com.respeito.C}(i) we immediately get that the number $\textrm{r.grade}_R(M, C)$ does not depend on $C$ in case $\textrm{grade}_R(M) >0$, but what if $\textrm{grade}_R(M) =0$\,? Below we provide classes of examples where the reduced grade turns out to depend on the choice of $C$. This, in particular, justifies the necessity of the presence of $C$ in the notation $\textrm{r.grade}_R(M, C)$, unlike what happens to the classical grade (see Lemma \ref{lema.grade.e.gradeC}).

\begin{example}\label{conductor}\rm Let $R$ be a 1-dimensional non-Gorenstein local domain possessing a canonical module $\omega_R$, and let $\overline{R}$ be the integral closure of $R$ in its fraction field $K$. Consider the ideal $\mathcal{I}=R :_R \overline{R}$, i.e.  the conductor of the extension $R\subset \overline{R}$. Since $R$ is not Gorenstein, \cite{Rego} yields $${\rm Ext}_R^1(\mathcal{I},\,R) \, \neq \, 0,$$ which gives ${\rm r.grade}_R(\mathcal{I}, R)={\rm r.grade}_R(\mathcal{I})=1$. Now, since $R$ is a 1-dimensional domain, the non-zero ideal $\mathcal{I}$ must be maximal Cohen-Macaulay as an $R$-module. Therefore, ${\rm Ext}_R^i(\mathcal{I}, \omega_R)  = 0$ for all $i>0$, which means ${\rm r.grade}_R(\mathcal{I}, \omega_R)=\infty$.

A similar class of examples over the same domain $R$ can be described as follows. Write $\mathcal{M}$ for the maximal ideal of $R$ and consider a non-zero finite submodule $$\mathcal{F} \, \subset \, \bigcup_{t\geq 0}\,(\mathcal{M}^t:_K\mathcal{M}^t),$$ 
i.e. $\mathcal{F}$ is just a fractional ideal
in the quadratic transform of $R$. By \cite[Corollary 3.2]{Ulrich}, we have ${\rm Ext}_R^1(\mathcal{F},\,R)  \neq  0$ and then ${\rm r.grade}_R(\mathcal{F}, R)=1$ while, as above, ${\rm r.grade}_R(\mathcal{F}, \omega_R)=\infty$.

In higher dimensions and reduced grade not necessarily equal to 1, a plentiful supply of examples can be produced by means of \cite[Theorem 3.1]{Ulrich}. It would be interesting to obtain an example where both (distinct) reduced grades are finite.

\end{example}

The next definition is crucial in this note.

\begin{definition}\label{def.rgrade.com.respeito.C}\rm
Let $M$ be a finite $R$-module with finite $\textrm{G}_C$-dimension. We
say that $M$ is \emph{reduced $\textrm{\emph{G}}_C$-perfect} if
$$\textrm{r.grade}_R(M, C) \, = \,
\textrm{G}_C\textrm{-dim}_R(M).$$
\end{definition}

This class of modules was introduced
in \cite{link2015}, as a generalization of the class of reduced G-perfect modules (see \cite{link2013}) which corresponds to the case $C=R$. Of course, every reduced $\textrm{G}_C$-perfect module has (finite) positive $\textrm{G}_C$-dimension.

It is worth mentioning a couple of notions that preceded the one recalled in Definition \ref{def.rgrade.com.respeito.C}. A finite $R$-module $M$ is said to be $\textrm{G}$\emph{-perfect}, or \emph{quasi-perfect}, if
$\textrm{grade}_R(M) = \textrm{G}\textrm{-dim}_R(M)$.
More generally, $M$ is $\textrm{G}_C$\emph{-perfect} if
$\textrm{grade}_R(M) = \textrm{G}_C\textrm{-dim}_R(M)$. We refer to \cite{foxby75} and \cite{golod84}.
Therefore, the concept of reduced $\textrm{G}_C$-perfection is an adaptation of $\textrm{G}_C$-perfection where $\textrm{grade}_R(M)$ is replaced with $\textrm{r.grade}_R(M, C)$, which makes sense by virtue of Remark \ref{obs.rgrade.com.respeito.C}(iii).

\begin{example} \label{ex1.GCperfeito.reduzido} \rm
Every $\textrm{G}_C$-perfect module of positive grade is reduced
$\textrm{G}_C$-perfect. Later on, in Example \ref{ex2.GCperfeito.reduzido}, we will show how to produce a family of reduced $\textrm{G}_C$-perfect
modules that are not $\textrm{G}_C$-perfect.
\end{example}

For an integer $n > 0$, we consider the composition
$$\mathcal{T}_{n}^C \, := \, \textrm{Tr}_C\Omega^{n-1},$$ i.e. to a given finite $R$-module $M$ it assigns the $R$-module $\textrm{Tr}_C(\Omega^{n-1}M)$.
Here, $\Omega^{n-1}$ is the $(n-1)$-th syzygy operator, which is defined recursively in the expected way: $\Omega^0M=M$, $\Omega^1M=\Omega M$, and $\Omega^{k}M=\Omega (\Omega^{k-1}M)$ for any $k\geq 2$, with all syzygy modules computed from a minimal projective resolution of $M$. In particular, $\mathcal{T}_{1}^C(M) = \textrm{Tr}_C(M)$.

Also note for completeness that if $P$ is a finite projective $R$-module, then ${\rm Tr}_C(P)$ is ${\rm add}_R(C)$-equivalent to zero; thus if for example $R$ is local, ${\rm Tr}_C(P)$ must be isomorphic to a direct sum of finitely many copies of $C$.

The following basic lemma will be useful throughout the paper.

\begin{lemma}\label{lema.chave.das.sequecias.exatas}
Let $M$ be a finite $R$-module and $n>0$ be an integer. Then, there
exist finite $R$-modules $L$ and $P$, with $P$ projective, fitting into short exact sequences $($for a fixed choice of transposes$)$
$$0 \rightarrow \textrm{\emph{Ext}}_{R}^n(M,\,C) \rightarrow
\mathcal{T}_{n}^C(M) \rightarrow L \rightarrow 0,$$
$$0 \rightarrow L \rightarrow \textrm{\emph{Tr}}_C(P) \rightarrow
\mathcal{T}_{n+1}^C(M) \rightarrow 0.$$ Moreover, if\,
$\textrm{\emph{r.grade}}_R(M, C) \geq n\geq 2$, there exists an
exact sequence $$0 \rightarrow \textrm{\emph{Tr}}_C(M)
\rightarrow P_2^{C} \rightarrow \cdots \rightarrow P_n^{C}
\rightarrow \mathcal{T}_{n}^C(M) \rightarrow 0,$$ where $P_i$ is
a finite projective $R$-module for all $i = 2, \ldots, n$.
\end{lemma}
\begin{proof}
Choose a short exact sequence
\begin{equation}\label{eq1.lema.chave.das.sequecias.exatas}
0 \rightarrow \Omega^nM \xrightarrow{f} P \rightarrow
\Omega^{n-1}M \rightarrow 0,
\end{equation}
where $P$ is a finite projective $R$-module. By \cite[Lemma
2.2]{link2015}, there is an exact sequence
\begin{equation}\label{eq2.lema.chave.das.sequecias.exatas}
0 \rightarrow (\Omega^{n-1}M)^{C} \rightarrow P^{C}
\xrightarrow{f^{C}}(\Omega^{n}M)^{C} \rightarrow \textrm{Tr}_C(\Omega^{n-1}M) \xrightarrow{\psi} \textrm{Tr}_C(P) \rightarrow
\textrm{Tr}_C(\Omega^{n}M) \rightarrow 0.
\end{equation}
Note that $\textrm{Tr}_C(\Omega^{n-1}M) = \mathcal{T}_{n}^C(M)$
and $\textrm{Tr}_C(\Omega^{n}M) = \mathcal{T}_{n+1}^C(M)$. On
the other hand, using
(\ref{eq1.lema.chave.das.sequecias.exatas}) and the fact that $\textrm{Ext}_{R}^1(\Omega^{n-1}M, C) \cong \textrm{Ext}_{R}^n(M, C)$, we obtain an exact
sequence
\begin{equation}\label{eq3.lema.chave.das.sequecias.exatas}
0 \rightarrow (\Omega^{n-1}M)^{C} \rightarrow P^{C}
\xrightarrow{f^{C}}(\Omega^{n}M)^{C} \rightarrow
\textrm{Ext}_{R}^n(M,\,C) \rightarrow 0.
\end{equation}
Therefore, by (\ref{eq2.lema.chave.das.sequecias.exatas}) and
(\ref{eq3.lema.chave.das.sequecias.exatas}), we get
$\textrm{Ext}_{R}^n(M, C) \cong \textrm{Ker}(\psi)$. Thus, we
have an exact sequence $$0 \rightarrow \textrm{Ext}_{R}^n(M,\,C)
\rightarrow \mathcal{T}_{n}^C(M) \xrightarrow{\psi}\textrm{Tr}_C(P) \rightarrow \mathcal{T}_{n+1}^C(M) \rightarrow 0.$$
Taking $L:= \textrm{Im}(\psi)$, we obtain the two desired exact sequences.

Now assume that $\textrm{r.grade}_R(M, C) \geq n \geq 2$, i.e.
$\textrm{Ext}_{R}^i(M, C)=0$ for all $0 < i < n$. By Remark \ref{obs.rgrade.com.respeito.C}(iii), we have G$_C$-${\rm dim}(M)\geq n$. In particular, any resolution of $M$ by finite projective $R$-modules must have length at least $n$. Then we can pick a minimal projective resolution
\begin{equation}\label{eq4.lema.chave.das.sequecias.exatas}
\cdots \rightarrow P_{n} \xrightarrow{\varphi_{n}} P_{n-1}
\rightarrow \cdots \rightarrow P_{1} \xrightarrow{\varphi_{1}}
P_0 \rightarrow M \rightarrow 0
\end{equation}
which in particular yields projective
presentations
$$P_{1} \xrightarrow{\varphi_{1}} P_0 \rightarrow M \rightarrow 0
\textrm{ ~~ and ~~ } P_{n} \xrightarrow{\varphi_{n}} P_{n-1}
\rightarrow \Omega^{n-1}M \rightarrow 0$$ of $M$ and $\Omega^{n-1}M$, respectively. We
get exact sequences
\begin{equation}\label{eq5.lema.chave.das.sequecias.exatas} 0 \rightarrow M^{C} \rightarrow P_0^{C}
\xrightarrow{\varphi_{1}^{C}} P_1^{C} \rightarrow \textrm{Tr}_C(M) \rightarrow 0,
\end{equation}
\begin{equation}\label{eq6.lema.chave.das.sequecias.exatas} 0 \rightarrow (\Omega^{n-1}M)^{C} \rightarrow P_{n-1}^{C}
\xrightarrow{\varphi_{n}^{C}} P_n^{C} \rightarrow
\mathcal{T}_{n}^C(M) \rightarrow 0.
\end{equation}
Applying $\textrm{Hom}_R (-, C)$ to (\ref{eq4.lema.chave.das.sequecias.exatas}), we obtain
a complex
\begin{equation}\label{eq7.lema.chave.das.sequecias.exatas}
0 \rightarrow M^{C} \rightarrow P_{0}^{C}
\xrightarrow{\varphi_{1}^{C}} P_{1}^{C}
\xrightarrow{\varphi_{2}^{C}} P_{2}^{C}
\rightarrow \cdots
\rightarrow P_{n-1}^{C}
\xrightarrow{\varphi_{n}^{C}}
 P_{n}^{C} \rightarrow \cdots.
\end{equation}
which is exact at least up to $P_{n-1}^{C}$ because
$\textrm{Ext}_{R}^i(M, C)=0$ for all $0 < i < n$. From the exact sequences
(\ref{eq5.lema.chave.das.sequecias.exatas})  and
(\ref{eq6.lema.chave.das.sequecias.exatas}), we have $\textrm{Tr}_C(M) \cong \textrm{Im}(\varphi_{2}^{C})$ and
$\mathcal{T}_{n}^C(M) \cong \textrm{Coker}(\varphi_{n}^{C})$.
Now the assertion is obvious by
(\ref{eq7.lema.chave.das.sequecias.exatas}). \qed
\end{proof}

\begin{proposition} \label{prop.EXT.TrC}
Let $M$ be a reduced $\textrm{\emph{G}}_C$-perfect $R$-module of
$\textrm{\emph{G}}_C$-dimension $n$. Then,
$$\textrm{\emph{Ext}}_{R}^i(\textrm{\emph{Tr}}_C (M),\,C) \, \cong \,
\textrm{\emph{Ext}}_{R}^{i+n-1}(\textrm{\emph{Ext}}_{R}^n (M,\,C),\,C)$$ for all $i >0$.
\end{proposition}
\begin{proof}
By Lemma \ref{lema.chave.das.sequecias.exatas}, there are exact
sequences
\begin{equation}\label{4.4.3}
0 \rightarrow \textrm{Ext}_{R}^n(M,\,C) \rightarrow
\mathcal{T}_{n}^C (M) \rightarrow L \rightarrow 0,
\end{equation}
\begin{equation}\label{4.4.4}
0 \rightarrow L \rightarrow \textrm{Tr}_C (P) \rightarrow
\mathcal{T}_{n+1}^C (M) \rightarrow 0,
\end{equation}
for certain finite $R$-modules $L$ and $P$, with $P$ projective. As $\textrm{G}_C\textrm{-dim}_R(M)=n$, it follows
that $\textrm{G}_C\textrm{-dim}_R(\Omega^n M)=0$. Then, using
Proposition \ref{prop.prop.2.6.geng},
$$\textrm{G}_C\textrm{-dim}_R(\mathcal{T}_{n+1}^C (M)) \, = \,
\textrm{G}_C\textrm{-dim}_R(\textrm{Tr}_C (\Omega^n M)) \, = \, 0$$ and,
in addition, $\textrm{G}_C\textrm{-dim}_R(\textrm{Tr}_C (P))=0$. Hence, by
(\ref{4.4.4}) and Proposition \ref{prop.item2.golod}, $\textrm{G}_C\textrm{-dim}_R(L)=0$. In
particular, $\textrm{Ext}_{R}^i(L, C) = 0$ for all $i > 0$. On the other hand, from
(\ref{4.4.3}) we derive a long exact sequence
$$\cdots \rightarrow \textrm{Ext}_{R}^i(L,\,C) \rightarrow
\textrm{Ext}_{R}^i(\mathcal{T}_{n}^C (M),\,C) \rightarrow
\textrm{Ext}_{R}^{i}(\textrm{Ext}_{R}^n (M,\,C),\,C) \rightarrow
\textrm{Ext}_{R}^{i+1}(L,\,C) \rightarrow \cdots.$$ Therefore,
\begin{equation}\label{eq.iso.Exts}
\textrm{Ext}_{R}^i(\mathcal{T}_{n}^C (M),\,C) \, \cong \,
\textrm{Ext}_{R}^{i}(\textrm{Ext}_{R}^n(M,\,C),\,C), \textrm{ for
all } i > 0.
\end{equation} Notice that this proves the proposition in the case $n=1$.
Thus we may assume that $n\geq 2$. In this situation, Lemma \ref{lema.chave.das.sequecias.exatas} gives an exact sequence
$$0 \rightarrow \textrm{Tr}_C(M) \rightarrow G_{n-2} \rightarrow
\cdots \rightarrow G_0 \rightarrow \mathcal{T}_{n}^C(M)
\rightarrow 0,$$
where $G_j := P_{n-j}^{C}$ and $P_{n-j}$ is a finite projective $R$-module for $j =0, \ldots, n-2$. By Proposition \ref{prop.prop.2.6.geng},
each $G_{j}$ has $\textrm{G}_C$-dimension zero and hence
$\textrm{Ext}_{R}^i(G_j, C) = 0$  for all $i > 0$. Now define $K_0 = \mathcal{T}_{n}^C(M)$, $K_1 =
\textrm{Ker} \, (G_0 \rightarrow K_0)$, and (if $n\geq 3$)
$K_{j+1} = \textrm{Ker} \, (G_{j} \rightarrow G_{j-1})$ for all $1
\leq j \leq n-2$. Note that $K_{n-1} \cong \textrm{Tr}_C (M)$. For a fixed $j \in \{0, \ldots, n-2 \}$, we have a
short exact sequence
$$0 \rightarrow K_{j+1} \rightarrow G_{j} \rightarrow K_{j}
\rightarrow 0,$$ which induces a long exact sequence $$\cdots
\rightarrow \textrm{Ext}_{R}^i(G_{j},\,C) \rightarrow
\textrm{Ext}_{R}^{i} (K_{j+1},\,C) \rightarrow
\textrm{Ext}_{R}^{i+1} (K_{j},\,C) \rightarrow
\textrm{Ext}_{R}^{i+1} (G_{j},\,C) \rightarrow \cdots.$$ As
$\textrm{Ext}_{R}^i (G_{j}, C) = 0 \textrm{ for all } i > 0$, we
get $\textrm{Ext}_{R}^{i} (K_{j+1}, C) \cong \textrm{Ext}_{R}^{i+1}
(K_{j}, C)$ for all $i > 0$. Therefore, for each $i >
0$, we have isomorphisms
$$\textrm{Ext}_R^i(K_{n-1},\,C) \, \cong \, \textrm{Ext}_R^{i+1}(K_{n-2},\,C) \,
\cong \, \cdots \, \cong \,
\textrm{Ext}_R^{i+n-1}(K_0,\,C).$$ Thus, $\textrm{Ext}_R^i(\textrm{Tr}_C (M), C)  \cong 
\textrm{Ext}_R^{i+n-1}(\mathcal{T}_{n}^C (M), C)$ for
all $i > 0$. Using (\ref{eq.iso.Exts}), the result follows. \qed
\end{proof}

\medskip

Theorem \ref{teo.Gperf.redu.formula.da.GCdim.generalizado} below is the main result of this section. It establishes a formula connecting the
$\textrm{G}_C$-dimension of a given reduced $\textrm{G}_C$-perfect
$R$-module $M$ with the $\textrm{G}_C$-dimensions of $\textrm{Tr}_C
(M)$ and $\textrm{Ext}_{R}^{\textrm{G}_C\textrm{-dim}_R(M)}(M,C)$.

\begin{theorem}\label{teo.Gperf.redu.formula.da.GCdim.generalizado}
Let $M$ be a reduced $\textrm{\emph{G}}_C$-perfect $R$-module of
$\textrm{\emph{G}}_C$-dimension $n$. Then,
$$\textrm{\emph{G}}_C\textrm{\emph{-dim}}_R(M) + \textrm{\emph{G}}_C\textrm{\emph{-dim}}_R(\textrm{\emph{Tr}}_C (M)) \, = \, \textrm{\emph{G}}_C\textrm{\emph{-dim}}_R(\textrm{\emph{Ext}}_{R}^n (M,\,C)) + 1.$$
\end{theorem}
\begin{proof} Let us begin with the case $n=1$. Recall that $\mathcal{T}_{1}^C(M) = \textrm{Tr}_C(M)$. Lemma \ref{lema.chave.das.sequecias.exatas} then gives a short exact sequence $0 \rightarrow \textrm{Ext}_{R}^1(M,\,C) \rightarrow
\textrm{Tr}_C(M) \rightarrow L_1 \rightarrow 0$ for a suitable finite $R$-module $L_1$ which must satisfy $\textrm{G}_C\textrm{-dim}_R(L_1) = 0$ (as seen in the proof of Proposition \ref{prop.EXT.TrC}). By Proposition \ref{prop.item6.golod}, we get that $\textrm{G}_C\textrm{-dim}_R(\textrm{Tr}_C (M)) < \infty$ if and only if $\textrm{G}_C\textrm{-dim}_R(\textrm{Ext}_{R}^1(M, C)) <
\infty$, and in this case Proposition \ref{prop.obs6.1.9.sather} ensures that these dimensions must coincide. This solves the case $n=1$.

Now suppose $n\geq 2$. According to Lemma \ref{lema.chave.das.sequecias.exatas}, there are exact
sequences
\begin{equation}\label{eq.finitude00}
0 \rightarrow \textrm{Tr}_C (M) \rightarrow G_{n-2} \rightarrow
\cdots \rightarrow G_0 \rightarrow \mathcal{T}_{n}^C (M)
\rightarrow 0,
\end{equation}
\begin{equation}\label{eq.finitude0}
0 \rightarrow \textrm{Ext}_{R}^n (M,\,C) \rightarrow
\mathcal{T}_{n}^C (M) \rightarrow L \rightarrow 0,
\end{equation}
where $\textrm{G}_C\textrm{-dim}_R(G_j)=0$ for all $j =0, \ldots,
n-2$, and moreover $\textrm{G}_C\textrm{-dim}_R(L) = 0$ by the proof of Proposition \ref{prop.EXT.TrC}.
Breaking (\ref{eq.finitude00}) into short exact sequences if necessary (i.e. if $n\geq 3$),
Proposition \ref{prop.item6.golod} gives $\textrm{G}_C\textrm{-dim}_R (\textrm{Tr}_C(M)) < \infty
\Leftrightarrow \textrm{G}_C\textrm{-dim}_R(\mathcal{T}_{n}^C(M)) < \infty$, and moreover by (\ref{eq.finitude0}) we get
$\textrm{G}_C\textrm{-dim}_R(\mathcal{T}_{n}^C (M)) < \infty
\Leftrightarrow \textrm{G}_C\textrm{-dim}_R(\textrm{Ext}_{R}^n(M, C)) < \infty$. It follows that
$$\textrm{G}_C\textrm{-dim}_R(\textrm{Tr}_C (M)) < \infty \, \Leftrightarrow \,
\textrm{G}_C\textrm{-dim}_R(\textrm{Ext}_{R}^n (M, C)) <
\infty.$$ Hence we can assume that these dimensions are finite. By Proposition \ref{prop.6.1.7.sather},
\begin{equation}\label{eq.formula.Gcdim1}
\textrm{G}_C\textrm{-dim}_R(\textrm{Ext}_{R}^n (M,\,C)) \, = \, \sup
\{ j \geq 0 \mid \textrm{Ext}_{R}^j(\textrm{Ext}_{R}^n (M,\,C),\,C)
\neq 0 \},
\end{equation}
\begin{equation}\label{eq.formula.Gcdim2}
\textrm{G}_C\textrm{-dim}_R(\textrm{Tr}_C (M)) \, = \, \sup \{ j
\geq 0 \mid \textrm{Ext}_{R}^j (\textrm{Tr}_C (M),\,C) \neq 0 \}.
\end{equation}
For any given $j>n-1$, we can apply Proposition \ref{prop.EXT.TrC} with $i=j-n+1$ in order to get
\begin{equation}\label{eq.iso.Exts2}
\textrm{Ext}_{R}^{j}(\textrm{Ext}_{R}^n (M,\,C),\,C) \, \cong \,
\textrm{Ext}_{R}^{j-n+1}(\textrm{Tr}_C (M),\,C).
\end{equation}
Now set $t = \textrm{G}_C\textrm{-dim}_R(\textrm{Tr}_C (M))$, which must be strictly positive because otherwise Proposition \ref{prop.prop.2.6.geng} would force
$n=0$, a contradiction. Taking $j = n+t-1 > n-1$ in (\ref{eq.iso.Exts2}), we obtain
$\textrm{Ext}_{R}^{n+t-1}(\textrm{Ext}_{R}^n (M, C), C)  \cong 
\textrm{Ext}_{R}^{t}(\textrm{Tr}_C(M), C)  \neq  0$, where the non-vanishing follows by
(\ref{eq.formula.Gcdim2}). For $j > n+t-1$, we have $j-n+1>t$, so (\ref{eq.formula.Gcdim2}) gives
$\textrm{Ext}_{R}^{j-n+1}(\textrm{Tr}_C (M), C) = 0$, which by (\ref{eq.iso.Exts2}) implies
$\textrm{Ext}_{R}^{j} (\textrm{Ext}_{R}^n(M, C), C)=0$. By (\ref{eq.formula.Gcdim1}), we get $\textrm{G}_C\textrm{-dim}_R(\textrm{Ext}_{R}^n(M, C))=n+t-1$, as needed. \qed
\end{proof}

\begin{corollary}\label{cor.Gcperf.redu.formula.do.depth.generalizado}
Let $R$ be a local ring and $M$ be a reduced
$\textrm{\emph{G}}_C$-perfect $R$-module of
$\textrm{\emph{G}}_C$-dimension $n$. Assume that
$\textrm{\emph{G}}_C\textrm{\emph{-dim}}_R(M^C) < \infty$. Then
$$\textrm{\emph{depth}}_R(M) + \textrm{\emph{depth}}_R(\textrm{\emph{Tr}}_C (M)) \, = \, \textrm{\emph{depth}}(R) +
\textrm{\emph{depth}}_R(\textrm{\emph{Ext}}_{R}^n (M,\,C)) - 1.$$
\end{corollary}
\begin{proof} Applying ${\rm Hom}_R(-, C)$ to a projective presentation $P_1\rightarrow P_0\rightarrow M\rightarrow 0$, we get an exact sequence
$0\rightarrow M^C\rightarrow P_0^C\rightarrow P_1^C\rightarrow \textrm{Tr}_C(M)\rightarrow0$, which by Proposition \ref{prop.item6.golod} implies that
$\textrm{G}_C\textrm{-dim}_R(\textrm{Tr}_C(M))< \infty$ if (and only if) $\textrm{G}_C\textrm{-dim}_R(M^C)< \infty$.
It follows by Theorem
\ref{teo.Gperf.redu.formula.da.GCdim.generalizado} that $\textrm{G}_C\textrm{-dim}_R(\textrm{Ext}_{R}^n (M, C)) < \infty$. Now set $r =\textrm{depth}(R)$. By Theorem
\ref{teo.propriedades.Gcdimensao} and Theorem
\ref{teo.Gperf.redu.formula.da.GCdim.generalizado}, we can write
\begin{eqnarray*}
\textrm{depth}_R(M)+\textrm{depth}_R(\textrm{Tr}_C (M)) & = & (r - \textrm{G}_C\textrm{-dim}_R(M))+ (r - \textrm{G}_C\textrm{-dim}_R(\textrm{Tr}_C (M))) \\
                                                     & = & 2r -(\textrm{G}_C\textrm{-dim}_R( M)+ \textrm{G}_C\textrm{-dim}_R(\textrm{Tr}_C (M))) \\
                                                     & = &  2r -\textrm{G}_C\textrm{-dim}_R (\textrm{Ext}_{R}^n (M,\,C))-1\\
                                                     & = & r + \textrm{depth}_R (\textrm{Ext}_{R}^n (M,\,C)) -1.
\end{eqnarray*}\qed
\end{proof}

\begin{remark}\label{obs.2.11.sadeghi}\rm Let $M$ be a finite $R$-module. Under an appropriate condition, we can relate the cohomology of $\textrm{{Tr}}_C(M)$
to that of $\lambda M \otimes_R C$. First, by the definition of the operator $\lambda$,
there is an exact sequence $0 \rightarrow \lambda M \rightarrow P\rightarrow \textrm{Tr}(M) \rightarrow 0$ for a suitable finite projective $R$-module $P$ and a fixed choice of transpose. By Proposition
\ref{lema.tr.e.trC}, $\textrm{Tr}(M) \otimes_R C \cong
\textrm{Tr}_C(M)$. Consequently, if we assume that $\textrm{{Tor}}^{R}_1
(\textrm{{Tr}}(M), C) = 0$, we have a short exact sequence
\begin{equation}\label{eq.seg.exata.lambidaM.e.TrC}
0 \rightarrow \lambda M \otimes_R C \rightarrow P \otimes_R C
\rightarrow \textrm{{Tr}}_C(M) \rightarrow 0.
\end{equation}
Now recall that
$\textrm{Ext}_{R}^i(P \otimes_R C, C) = 0$ for all $i > 0$, because
$\textrm{G}_C\textrm{-dim}_R(P\otimes_RC) = 0$ as seen in the proof of Lemma \ref{obs.syzygy}(ii).
Hence, by (\ref{eq.seg.exata.lambidaM.e.TrC}), we get isomorphisms
$$\textrm{{Ext}}_{R}^i(\lambda M \otimes_R C,\,C) \, \cong \,
\textrm{{Ext}}_{R}^{i+1}(\textrm{{Tr}}_C(M),\,C),$$ for all $i > 0.$
\end{remark}

The following result is a generalization of \cite[Proposition
3.5(i)]{link2013}.

\begin{corollary}\label{cor.EXT.Lambda}
Let $M$ be a reduced $\textrm{\emph{G}}_C$-perfect $R$-module of $\textrm{\emph{G}}_C$-dimension $n$. Assume that
$\textrm{\emph{Tor}}^{R}_1 (\textrm{\emph{Tr}}(M), C) = 0$. Then
$$\textrm{\emph{Ext}}_{R}^i(\lambda M \otimes_R C,\,C) \, \cong \,
\textrm{\emph{Ext}}_{R}^{i+n}(\textrm{\emph{Ext}}_{R}^n(M,\,C),\,C),$$ for all $i
>0$.
\end{corollary}
\begin{proof}
This is an immediate consequence of Proposition \ref{prop.EXT.TrC}
and Remark \ref{obs.2.11.sadeghi}. \qed
\end{proof}

\begin{corollary}\label{cor.Gperf.redu.formula.da.GCdim.com.Lambda}
Let $M$ be a reduced $\textrm{\emph{G}}_C$-perfect $R$-module of $\textrm{\emph{G}}_C$-dimension $n$. Assume that
$\textrm{\emph{Tor}}^{R}_1(\textrm{\emph{Tr}}(M), C) = 0$. Then
$$\textrm{\emph{G}}_C\textrm{\emph{-dim}}_R(M) + \textrm{\emph{G}}_C\textrm{\emph{-dim}}_R(\lambda M \otimes_R C) \, = \, \textrm{\emph{G}}_C\textrm{\emph{-dim}}_R
(\textrm{\emph{Ext}}_{R}^n (M,\,C)).$$
\end{corollary}

\begin{proof}
By Theorem \ref{teo.Gperf.redu.formula.da.GCdim.generalizado}, we have
\begin{equation}\label{eq.blabla1}
\textrm{G}_C\textrm{-dim}_R(M) + \textrm{G}_C\textrm{-dim}_R
(\textrm{Tr}_C (M)) - 1 \, = \, \textrm{G}_C\textrm{-dim}_R
(\textrm{Ext}_{R}^n (M,\,C)).
\end{equation}
As $\textrm{G}_C\textrm{-dim}_R(P \otimes_R C) = 0$, Remark
\ref{obs.2.11.sadeghi} and Proposition \ref{prop.item6.golod} give
$$\textrm{G}_C\textrm{-dim}_R(\lambda M \otimes_R C) < \infty
\, \Leftrightarrow \, \textrm{G}_C\textrm{-dim}_R (\textrm{Tr}_C (M))
< \infty.$$ Therefore, if $\textrm{G}_C\textrm{-dim}_R (\lambda
M \otimes_R C) = \infty$, then $\textrm{G}_C\textrm{-dim}_R
(\textrm{Tr}_C (M)) = \infty$, which by (\ref{eq.blabla1}) implies that
$\textrm{G}_C\textrm{-dim}_R
(\textrm{Ext}_{R}^n (M, C)) = \infty$ and the claim follows. Now, if $\textrm{G}_C\textrm{-dim}_R(\lambda M
\otimes_R C) < \infty$ then $\textrm{G}_C\textrm{-dim}_R(\textrm{Tr}_C (M)) <
\infty$, and since $\textrm{G}_C\textrm{-dim}_R(M)
> 0$, Proposition \ref{prop.prop.2.6.geng} forces $\textrm{G}_C\textrm{-dim}_R(\textrm{Tr}_C (M))>0$. By Proposition \ref{prop.lema.1.9.gerko} and the exact sequence of Remark \ref{obs.2.11.sadeghi},
$$\textrm{G}_C\textrm{-dim}_R(\lambda M \otimes_R C) \, = \,
\textrm{G}_C\textrm{-dim}_R(\textrm{Tr}_C(M)) - 1,$$
and the assertion now follows by (\ref{eq.blabla1}). \qed
\end{proof}

\begin{corollary}\label{cor.cor.Gcperf.redu.formula.do.depth.generalizado}
Let $R$ be a local ring and $M$ be a reduced
$\textrm{\emph{G}}_C$-perfect $R$-module of
$\textrm{\emph{G}}_C$-dimension $n$. Assume that
$\textrm{\emph{G}}_C\textrm{\emph{-dim}}_R(M^C) < \infty$ and $\textrm{\emph{Tor}}^{R}_1 (\textrm{\emph{Tr}}(M),
C) = 0$. Then
$$\textrm{\emph{depth}}_R(M) + \textrm{\emph{depth}}_R (\lambda M \otimes_R C) \, = \, \textrm{\emph{depth}}(R) +
\textrm{\emph{depth}}_R(\textrm{\emph{Ext}}_{R}^n (M,\,C)).$$
\end{corollary}
\begin{proof} The condition $\textrm{G}_C\textrm{-dim}_R(M^C)< \infty$ is equivalent to $\textrm{G}_C\textrm{-dim}_R(\textrm{Tr}_C(M))< \infty$ (see the proof of Corollary \ref{cor.Gcperf.redu.formula.do.depth.generalizado}), and this in turn yields, by Corollary \ref{cor.Gperf.redu.formula.da.GCdim.com.Lambda} and its proof, that $\textrm{G}_C\textrm{-dim}_R(\lambda M \otimes_R C)<
\infty$ and $\textrm{G}_C\textrm{-dim}_R(\textrm{Ext}_{R}^n (M, C)) < \infty$. Now the result follows easily by Corollary
\ref{cor.Gperf.redu.formula.da.GCdim.com.Lambda} combined with Theorem \ref{teo.propriedades.Gcdimensao}. \qed
\end{proof}

\medskip

In the next corollary, we get rid of the hypothesis $\textrm{{Tor}}^{R}_1 (\textrm{{Tr}}(M),
C) = 0$ and derive a formula for the depth of $\lambda M$, so as to give, in particular, a criterion for the freeness of this module. 

\begin{corollary}\label{cor.cor.Gcperf.redu.formula.do.depth.generalizado2}
Let $R$ be a Cohen-Macaulay local ring and $M$ be a reduced
$\textrm{\emph{G}}_C$-perfect $R$-module of
$\textrm{\emph{G}}_C$-dimension $n$. Assume that $M^C$ has finite ${\rm G}_C$-dimension and that $M^*$ has finite projective dimension. Then
$$\textrm{{\rm depth}}_R(\lambda M) \, = \, n + \textrm{{\rm depth}}_R(\textrm{{\rm Ext}}_{R}^n (M,\,C)).$$ In particular, $\lambda M$ is a free $R$-module if and only if\, $\textrm{{\rm depth}}_R(\textrm{{\rm Ext}}_{R}^n(M, C))={\rm dim}(R)-n$.
\end{corollary}
\begin{proof} Since $M^*\cong \Omega^2{\rm Tr}(M)$, the $R$-module ${\rm Tr}(M)$ (and $\lambda M$ as well) must also have finite projective dimension. Moreover, because $R$ is Cohen-Macaulay, it follows by \cite[Theorem 2.2.6(c)]{sather} that $C$ is maximal Cohen-Macaulay. Applying \cite[Lemma 2.2]{Yoshida} we get $\textrm{{Tor}}^{R}_i (\textrm{{Tr}}(M), C)  = 0$ for all $i>0$. Note that, by Theorem \ref{teo.propriedades.Gcdimensao}, we have ${\rm depth}_R(M)={\rm depth}(R)-n$. Therefore, by Corollary \ref{cor.cor.Gcperf.redu.formula.do.depth.generalizado}, we can write $$\textrm{{depth}}_R(\lambda M \otimes_R C) \, = \, n + \textrm{{depth}}_R(\textrm{{Ext}}_{R}^n (M,\,C)).$$ On the other hand, by definition, there is a short exact sequence $0\rightarrow \lambda M\rightarrow P\rightarrow \textrm{{Tr}}(M)\rightarrow 0$ for some finite projective $R$-module $P$. Tensoring with $C$, we get $\textrm{{Tor}}^{R}_j(\lambda M, C)  = 0$ for all $j>0$, which by \cite[Theorem 1.2]{Auslander} (together with the Auslander-Buchsbaum formula) gives 
$$\textrm{{depth}}_R(\lambda M \otimes_R C) \, = \, \textrm{{depth}}_R(\lambda M) + {\rm depth}_R(C) - {\rm depth}(R) \, = \, \textrm{{depth}}_R(\lambda M).$$ The result follows, and the particular assertion is clear.
\qed
\end{proof}

\begin{remark}\rm Maintain the setting and hypotheses of Corollary \ref{cor.cor.Gcperf.redu.formula.do.depth.generalizado2}. If moreover $\lambda M$ is not free, then, since $M^*=\Omega \lambda M$, we easily derive a formula for the depth of $M^*$:
$$\textrm{{\rm depth}}_R(M^*) \, = \, 1 + n + \textrm{{\rm depth}}_R(\textrm{{\rm Ext}}_{R}^n (M,\,C)).$$
\end{remark}

\section{Reduced $\textrm{G}_C$-perfection and (relative) Auslander
transpose}\label{sec32}

We will extend some results of \cite{AB} to the context of
$\textrm{G}_C$-dimension and present formulas relating the
grade and the reduced grade with respect to $C$. We will also investigate conditions under which
the reduced $\textrm{G}_C$-perfect property is preserved under the operation of taking
Auslander transpose with respect to $C$. As a main consequence, we will generalize \cite[Corollary 3.6]{link2013}.

The following basic result is a generalization of \cite[Proposition 4.17]{AB}.

\begin{proposition} \label{prop.4.16.stablemodule.generalizada}
Let $M$ be a finite $R$-module of finite
$\textrm{\emph{G}}_C$-dimension. Then, for any given $i > 0$,
$$\textrm{\emph{grade}}_R(\textrm{\emph{Ext}}_{R}^i(M,\,C)) \, \geq \, i.$$
\end{proposition}
\begin{proof} First, the hypothesis $\textrm{G}_{C}\textrm{-dim}_{R}(M)<\infty$ is equivalent to $\textrm{G}_{C_{\mathfrak
p}}\textrm{-dim}_{R_{\mathfrak p}}(M_{\mathfrak p}) < \infty$
for all ${\mathfrak p} \in \textrm{Spec}(R)$ (see \cite[Corollary 3.4]{geng2013}). Pick $i>0$. If $\textrm{Ext}_{R}^i(M, C) = 0$ then there is nothing
to prove, since by convention $\textrm{grade}_R (0) = \infty$. Thus we may suppose
$\textrm{Ext}_{R}^i(M, C) \neq 0$. According to \cite[Corollary
4.6]{AB}, we can write
\begin{equation}\label{eq.cor.4.6}
\textrm{grade}_R(\textrm{Ext}_{R}^i(M,\,C)) \, = \, \min \{\textrm{depth}(R_{\mathfrak
p}) \mid {\mathfrak p} \in \textrm{Supp}_R (\textrm{Ext}_{R}^i(M,\,C))\}.
\end{equation}
For an arbitrary ${\mathfrak p} \in \textrm{Supp}_R (\textrm{Ext}_{R}^i(M, C))$, we will show that $\textrm{depth}
(R_{\mathfrak p}) \geq i$ and thus by (\ref{eq.cor.4.6}) it
will follow that ${\rm grade}_R(\textrm{Ext}_{R}^i (M, C))\geq i$. Since
$\textrm{G}_{C_{\mathfrak p}}\textrm{-dim}_{R_{\mathfrak p}} (M_{\mathfrak p}) < \infty$ and
$$\textrm{Ext}_{R_{\mathfrak p}}^i(M_{\mathfrak p},\,C_{\mathfrak
p}) \, \cong \, (\textrm{Ext}_{R}^i(M,\,C))_{\mathfrak p} \, \neq \, 0,$$ we must have $\textrm{G}_{C_{\mathfrak
p}}\textrm{-dim}_{R_{\mathfrak p}}(M_{\mathfrak p}) \geq i$ by Proposition \ref{prop.6.1.7.sather}.
Therefore, Theorem \ref{teo.propriedades.Gcdimensao} yields
$$\textrm{depth}(R_{\mathfrak p}) \, = \, \textrm{G}_{C_{\mathfrak
p}}\textrm{-dim}_{R_{\mathfrak p}}(M_{\mathfrak p}) +
\textrm{depth}_{R_{\mathfrak p}}(M_{\mathfrak p}) \, \geq \,
\textrm{G}_{C_{\mathfrak p}}\textrm{-dim}_{R_{\mathfrak p}} (M_{\mathfrak p}) \, \geq \, i,$$ as desired. \qed
\end{proof}

\begin{proposition}\label{prop.Gperf.redu.formula.do.rgrade.e.gradeC} Let $M$ be a reduced $\textrm{\emph{G}}_C$-perfect $R$-module of $\textrm{\emph{G}}_C$-dimension $n$. Then,
$$\textrm{\emph{r.grade}}_R(M, C) + \textrm{\emph{r.grade}}_R(\textrm{\emph{Tr}}_C (M), C) \, = \, \textrm{\emph{grade}}_R(\textrm{\emph{Ext}}_{R}^n (M,\,C)) + 1.$$
\end{proposition}
\begin{proof} Let $j>n-1$ be an integer. Applying Proposition \ref{prop.EXT.TrC}
with $i = j-n+1$, we get
\begin{equation}\label{eq.iso.Exts3}
\textrm{Ext}_{R}^{j}(\textrm{Ext}_{R}^n(M,\,C),\,C) \, \cong \,
\textrm{Ext}_{R}^{j-n+1}(\textrm{Tr}_C (M),\,C).
\end{equation}
Assume that $\textrm{r.grade}_R(\textrm{Tr}_C (M), C) <
\infty$. Set $m = \textrm{r.grade}_R(\textrm{Tr}_C (M), C)$.
Taking $j = n+m-1 > n-1$ in (\ref{eq.iso.Exts3}), we get
$$\textrm{Ext}_{R}^{n+m-1}(\textrm{Ext}_{R}^n (M,\,C),\,C) \, \cong \,
\textrm{Ext}_{R}^{m} (\textrm{Tr}_C (M),\,C) \, \neq \, 0,$$ which gives $\textrm{{grade}}_R(\textrm{{Ext}}_{R}^n (M,
C))\leq n+m-1$, and we must show that equality holds. To this end, let $ 0 \leq
k < n+m-1$. Consider first the case $ n - 1 < k < n+m-1$, i.e. $ 1 \leq k - n + 1 < m$. By
(\ref{eq.iso.Exts3}) it follows that
$\textrm{Ext}_{R}^{k}(\textrm{Ext}_{R}^n (M, C), C) =  0$. Now suppose $0 \leq k < n$ and note that, by Proposition \ref{prop.4.16.stablemodule.generalizada},
$\textrm{grade}_R(\textrm{Ext}_{R}^n (M, C)) \geq n$, which implies $\textrm{Ext}_{R}^{k}
(\textrm{Ext}_{R}^n (M, C), C) = 0$ by virtue of
Lemma \ref{lema.grade.e.gradeC}. Thus we have shown that $\textrm{grade}_R(\textrm{Ext}_{R}^n (M, C))=n+m-1$.

Finally, if $\textrm{r.grade}_R(\textrm{Tr}_C (M), C) =
\infty$, then $\textrm{Ext}_{R}^{i} (\textrm{Tr}_C (M), C) = 0$
for all $i > 0$. By (\ref{eq.iso.Exts3}), $$\textrm{Ext}_{R}^{j}(\textrm{Ext}_{R}^n (M,\,C),\,C)
\, = \, 0, \textrm{
for all } j \geq n.$$ As $\textrm{grade}_R(\textrm{Ext}_{R}^n
(M, C)) \geq n$, it follows that $\textrm{Ext}_{R}^{j}
(\textrm{Ext}_{R}^n(M, C), C) = 0$ also for $0 \leq j < n$.
Therefore, $\textrm{grade}_R(\textrm{Ext}_{R}^n (M, C)) =
\infty$. The result follows. \qed
\end{proof}

\begin{lemma}\label{lema.formula.do.rgrade.geral} Given an integer $k>0$, let $0 \rightarrow N \rightarrow X_k \rightarrow \cdots \rightarrow X_1 \rightarrow M \rightarrow 0$ be an exact sequence of finite $R$-modules such that $\textrm{\emph{Ext}}_{R}^i (X_j, C)=0$
for all $i> 0$ and $j = 1, \ldots, k$. If $\textrm{\emph{r.grade}}_R(M, C)>k$,
then $$\textrm{\emph{r.grade}}_R(N, C) \, = \, \textrm{\emph{r.grade}}_R(M, C)-k.$$
\end{lemma}
\begin{proof}
We proceed by induction on $k$. Assuming $k = 1$, there are isomorphisms
\begin{equation}\label{eq1.lema.formula.do.rgrade.geral}
\textrm{Ext}_{R}^i (N,\,C) \, \cong \, \textrm{Ext}_{R}^{i+1} (M,\,C),
\textrm{ for all } i > 0.
\end{equation}
First, if $\textrm{r.grade}_R(M, C) = \infty$ then, equivalently,
$\textrm{Ext}_{R}^{j} (M, C)=0$ for all $j> 0$ (see Remark \ref{obs.rgrade.com.respeito.C}(ii)), and then (\ref{eq1.lema.formula.do.rgrade.geral}) gives
$\textrm{Ext}_{R}^i (N, C)=0$ for all $i>0$, so that $\textrm{r.grade}_R(N, C) = \infty$.

If $\textrm{r.grade}_R(M, C) < \infty$, say $t := \textrm{r.grade}_R(M, C)$ (note that $t>1$), then
(\ref{eq1.lema.formula.do.rgrade.geral}) yields
\begin{equation}\label{eq2.lema.formula.do.rgrade.geral}
\textrm{Ext}_{R}^{t-1} (N,\,C) \, \cong \, \textrm{Ext}_{R}^{t} (M,\,C) \, \neq \, 0,
\end{equation} hence
$\textrm{r.grade}_R(N, C) \leq t-1$. In case $t=2$, we have $\textrm{r.grade}_R(N, C) =  1 = t-1$. If $t\geq 3$, 
let $0 < i < t-1$, i.e. $1 < i+1 < t$. Using
(\ref{eq1.lema.formula.do.rgrade.geral}), we get
$\textrm{Ext}_{R}^i (N, C) \cong \textrm{Ext}_{R}^{i+1} (M, C)=0$. Thus, by (\ref{eq2.lema.formula.do.rgrade.geral}), $\textrm{r.grade}_R(N, C) = t-1$.

Now suppose $k > 1$. Set $K = \ker \,(X_{k-1} \rightarrow
X_{k-2})$. Then we have an exact sequence $0 \rightarrow K
\rightarrow X_{k-1} \rightarrow \cdots \rightarrow X_1 \rightarrow
M \rightarrow 0$. By induction,
\begin{equation}\label{eq4.lema.formula.do.rgrade.geral}
\textrm{r.grade}_R(K, C) \, = \, \textrm{r.grade}_R(M, C)-k+1.
\end{equation}
By what we already proved, the exact sequence $0 \rightarrow N \rightarrow X_{k} \rightarrow K \rightarrow 0$ gives an equality
$\textrm{r.grade}_R(N, C) = \textrm{r.grade}_R(K, C)-1$. Substituting into (\ref{eq4.lema.formula.do.rgrade.geral}), we get the result.
\qed
\end{proof}

\medskip

The lemma above will be particularly useful in the proof of the following example, the first part of which generalizes
\cite[Example 3.2]{link2013} (which treats the classical case $C=R$).

\begin{example} \label{ex2.GCperfeito.reduzido} \rm
Let $n$ be a positive integer and let $M$ be a non-zero finite $R$-module such that $\textrm{grade}_R(M) \geq n$. Then:
\begin{itemize}
  \item[(i)] $\mathcal{T}^C_{n} (M)$ is a reduced
  $\textrm{G}_C$-perfect $R$-module of $\textrm{G}_C$-dimension $n$ (or projective dimension $n$ if $C=R$);  
  \item[(ii)] If $R$ is local and $\textrm{grade}_R(M) > n$, then $\mathcal{T}^C_{n}(M)$
  is not a $\textrm{G}_C$-perfect $R$-module, and $\textrm{grade}_R(\mathcal{T}^C_{n} (M))=0$.
 \end{itemize}
\end{example}
\begin{proof}
(i) We prove the claim by induction on $n$. If $n = 1$, then
$\textrm{grade}_R(M) \geq 1$, which implies that $M^{C} = 0$.
A minimal projective presentation $P_1 \rightarrow P_0 \rightarrow M \rightarrow 0$ of $M$ yields a short exact sequence $0 \rightarrow P_0^{C} \rightarrow
P_1^{C} \rightarrow \textrm{Tr}_C (M) \rightarrow 0$, so that
$\textrm{G}_C\textrm{-dim}_R(\textrm{Tr}_C (M)) \leq 1$. Note that
$\textrm{G}_C\textrm{-dim}_R(M)$ is positive since, by Lemma \ref{lema.grade.e.gradeC}, $\textrm{G}_C\textrm{-dim}_R(M)\geq {\rm grade}_R(M)\geq n>0$. Hence $\textrm{G}_C\textrm{-dim}_R(\textrm{Tr}_C (M))>0$ by virtue of Proposition \ref{prop.prop.2.6.geng}(iii). It follows, by Remark
\ref{obs.rgrade.com.respeito.C}(iii), that the $R$-module $\mathcal{T}^C_{1} (M) = \textrm{Tr}_C (M)$ is reduced
$\textrm{G}_C$-perfect of $\textrm{G}_C$-dimension $1$.

Now assume that $n > 1$. Applying the induction hypothesis to
$\mathcal{T}^C_{n-1} (M)$ we obtain that $\mathcal{T}^C_{n-1} (M)$
is a reduced $\textrm{G}_C$-perfect $R$-module of
$\textrm{G}_C$-dimension $n-1$. As $\textrm{Ext}_{R}^{n-1} (M, C)
= 0$, Lemma \ref{lema.chave.das.sequecias.exatas} gives a short
exact sequence
\begin{equation}\label{eq0.ex2.GCperfeito.reduzido}
0 \rightarrow \mathcal{T}^C_{n-1} (M) \rightarrow \textrm{Tr}_C
(P) \rightarrow \mathcal{T}^C_{n} (M) \rightarrow 0.
\end{equation}
Because $\textrm{G}_C\textrm{-dim}_R(\mathcal{T}^C_{n-1} (M)) = n-1$ and $\textrm{G}_C\textrm{-dim}_R(\textrm{Tr}_C (P)) = 0$, this sequence yields $\textrm{G}_C\textrm{-dim}_R (\mathcal{T}^C_{n} (M)) \leq n$. To conclude, 
by virtue of Remark
\ref{obs.rgrade.com.respeito.C}(iii)
it suffices to show that
$\textrm{r.grade}_R(\mathcal{T}^C_{n} (M), C) \geq n$ and
$\textrm{G}_C\textrm{-dim}_R(\mathcal{T}^C_{n} (M)) > 0$.
Since $\Omega^{n-1}M$ is an $R$-syzygy module, Lemma
\ref{obs.syzygy} forces $\textrm{Ext}_{R}^1(\mathcal{T}^C_{n}(M),
C) = 0$ and thus $\textrm{r.grade}_R \, (\mathcal{T}^C_{n} (M), C) > 1$. Applying Lemma \ref{lema.formula.do.rgrade.geral} to (\ref{eq0.ex2.GCperfeito.reduzido}), we get
$\textrm{r.grade}_R(\mathcal{T}^C_{n} (M), C) = n$. Notice, finally, that if we assume $\textrm{G}_C\textrm{-dim}_R(\mathcal{T}^C_{n} (M)) = 0$ then $\textrm{G}_C\textrm{-dim}_R(\Omega^{n-1}M) = 0$ (by Proposition \ref{prop.prop.2.6.geng}(iii)) and this implies that $M$ has
$\textrm{G}_C$-dimension less than $n$, contradicting the
assumption that $\textrm{grade}_R(M) \geq n$ (see Lemma \ref{lema.grade.e.gradeC}). 

The assertion concerning projective dimension is clear from the proof with $C=R$.

\medskip

(ii) For simplicity, set $M' = \mathcal{T}^C_{n} (M)$. By part (i), $M'$ is reduced
 $\textrm{G}_C$-perfect of $\textrm{G}_C$-dimension $n$, hence
Proposition
\ref{prop.Gperf.redu.formula.do.rgrade.e.gradeC} gives
\begin{equation}\label{eq1.ex2.GCperfeito.reduzido}
\textrm{r.grade}_R(M', C) + \textrm{r.grade}_R(\textrm{Tr}_C (M'), C) \, = \, \textrm{grade}_R(\textrm{Ext}_{R}^n
(M',\,C)) + 1.
\end{equation}
By Proposition \ref{prop.lemma.2.12.paper5}, there exists an exact sequence
$0 \rightarrow \Omega^{n-1} M \rightarrow \textrm{Tr}_C (M')
\rightarrow X \rightarrow 0$,
where $\textrm{G}_C\textrm{-dim}_R(X)=0$. In particular,
$\textrm{Ext}_{R}^i (X, C) = 0$ for all $i >0$, which implies that
$\textrm{Ext}_{R}^i (\Omega^{n-1} M, C) \cong \textrm{Ext}_{R}^{i}
(\textrm{Tr}_C (M'), C)$, for all  $i > 0$. Thus,
\begin{equation}\label{eq2.ex2.GCperfeito.reduzido}
\textrm{r.grade}_R(\Omega^{n-1} M, C) \, = \, \textrm{r.grade}_R(\textrm{Tr}_C (M'), C).
\end{equation}
On the other hand, Lemma
\ref{lema.formula.do.rgrade.geral} (with $k=n-1$) yields
\begin{equation}\label{eq3.ex2.GCperfeito.reduzido}
\textrm{r.grade}_R(\Omega^{n-1} M, C) \, = \, \textrm{r.grade}_R(M, C) - n + 1.
\end{equation}
By (\ref{eq1.ex2.GCperfeito.reduzido}),
(\ref{eq2.ex2.GCperfeito.reduzido}),
(\ref{eq3.ex2.GCperfeito.reduzido}), and the equality $\textrm{r.grade}_R(M', C)=n$, we have
\begin{equation}\label{eq4.ex2.GCperfeito.reduzido}
\textrm{r.grade}_R(M, C) \, = \, \textrm{grade}_R(\textrm{Ext}_{R}^n (M',\,C)).
\end{equation}
Now assume that $M'$ is $\textrm{G}_C$-perfect.
So, it has grade $n$. By \cite[10., p.\,68]{golod84} (which requires $R$ to be local), the module $\textrm{Ext}_{R}^n
(M', C)$ must have grade $n$ as well. Therefore, by (\ref{eq4.ex2.GCperfeito.reduzido}), $\textrm{r.grade}_R(M,
C) = n$ which forces $\textrm{grade}_R(M) = n$, a contradiction. Hence indeed $M'$ is not $\textrm{G}_C$-perfect. Clearly, because of (i), this would be violated if $\textrm{grade}_R(M') > 0$. 
\qed
\end{proof}

\medskip

The following theorem investigates when the property of reduced
$\textrm{G}_C$-perfection is preserved under the operation of taking Auslander transpose with respect to $C$.

\begin{theorem}\label{teo.inspirado.no.cor.3.16.paper3}
Let $M$ be a finite $R$-module. Let $n$ and $t$ be two integers. Then the following statements are equivalent:
\begin{itemize}
  \item[(i)] $M$ is reduced $\textrm{\emph{G}}_C$-perfect of $\textrm{\emph{G}}_C$-dimension $n$ and $\textrm{\emph{Ext}}_{R}^n
(M, C)$ is $\textrm{\emph{G}}_C$-perfect of
$\textrm{\emph{G}}_C$-dimension $n+t-1$;
  \item[(ii)] $\textrm{\emph{Tr}}_C (M)$ is reduced $\textrm{\emph{G}}_C$-perfect of $\textrm{\emph{G}}_C$-dimension $t$ and $\textrm{\emph{Ext}}_{R}^t
(\textrm{\emph{Tr}}_C (M), C)$ is $\textrm{\emph{G}}_C$-perfect of
$\textrm{\emph{G}}_C$-dimension $n+t-1$.
\end{itemize}
\end{theorem}
\begin{proof}
Set $N = \textrm{Tr}_C (M)$. By Proposition
\ref{prop.lemma.2.12.paper5}, we have a short exact sequence
\begin{equation}\label{eq.seq.exata.lema2.12.paper5}
0 \rightarrow M \rightarrow \textrm{Tr}_C (N) \rightarrow X
\rightarrow 0
\end{equation}
where $\textrm{G}_C\textrm{-dim}_R(X)=0$. As
$\textrm{Ext}_{R}^i (X, C) = 0$ for all $i >0$, the
sequence (\ref{eq.seq.exata.lema2.12.paper5}) gives $\textrm{Ext}_{R}^i (M, C) \cong
\textrm{Ext}_{R}^{i} (\textrm{Tr}_C (N), C) \textrm{ for all } i
> 0,$ so that
\begin{equation}\label{eq.isomorfismos.EXTs.TrC.TrC}
\textrm{r.grade}_R(\textrm{Tr}_C (N), C) \, = \, \textrm{r.grade}_R(M, C).
\end{equation}
Applying Proposition
\ref{prop.obs6.1.9.sather} to
(\ref{eq.seq.exata.lema2.12.paper5}), we get
\begin{equation}\label{eq2.isomorfismos.EXTs.TrC.TrC}
\textrm{G}_C\textrm{-dim}_R(\textrm{Tr}_C (N)) \, = \,
\textrm{G}_C\textrm{-dim}_R(M).
\end{equation}

We are now prepared to prove
(i)$\Rightarrow$(ii). By Theorem
\ref{teo.Gperf.redu.formula.da.GCdim.generalizado} and Proposition
\ref{prop.Gperf.redu.formula.do.rgrade.e.gradeC}, we have
\begin{eqnarray*}
\textrm{G}_C\textrm{-dim}_R(N) & = & \textrm{G}_C\textrm{-dim}_R (\textrm{Ext}_{R}^n (M,\,C)) + 1 - \textrm{G}_C\textrm{-dim}_R(M) \\
                        & = & \textrm{grade}_R (\textrm{Ext}_{R}^n (M, C)) + 1 - \textrm{r.grade}_R(M, C) \\
                        & = & \textrm{r.grade}_R(N, C),
\end{eqnarray*}
and hence $N$ is reduced $\textrm{G}_C$-perfect of
$\textrm{G}_C$-dimension equal to $(n+t-1)+1-n=t$. Thus, similarly,
\begin{equation}\label{eq.GC.dimensao.formula}
\textrm{G}_C\textrm{-dim}_R(N) + \textrm{G}_C\textrm{-dim}_R
(\textrm{Tr}_C (N)) \, = \, \textrm{G}_C\textrm{-dim}_R(\textrm{Ext}_{R}^t (N,\,C))+1
\end{equation} and
\begin{equation}\label{eq.r.grade.c.formula}
\textrm{r.grade}_R(N, C) + \textrm{r.grade}_R(\textrm{Tr}_C (N), C) \, = \, \textrm{grade}_R(\textrm{Ext}_{R}^t
(N,\,C))+1.
\end{equation}
By (\ref{eq.isomorfismos.EXTs.TrC.TrC}),
(\ref{eq2.isomorfismos.EXTs.TrC.TrC}),
(\ref{eq.GC.dimensao.formula}) and (\ref{eq.r.grade.c.formula}),
we get
\begin{eqnarray*}
\textrm{G}_C\textrm{-dim}_R(\textrm{Ext}_{R}^t (N,\,C)) & = & \textrm{G}_C\textrm{-dim}_R(N) + n -1\\
                        & = & \textrm{r.grade}_R(N, C) + n - 1\\
                        & = & \textrm{grade}_R (\textrm{Ext}_{R}^t (N,\,C)).
\end{eqnarray*}
Therefore, $\textrm{Ext}_{R}^t (N, C)$ is $\textrm{G}_C$-perfect
of $\textrm{G}_C$-dimension $t+n-1$.

Now let us show the implication (ii)$\Rightarrow$(i). By Theorem
\ref{teo.Gperf.redu.formula.da.GCdim.generalizado} and Proposition
\ref{prop.Gperf.redu.formula.do.rgrade.e.gradeC}, we get
\begin{equation}\label{eq2.GC.dimensao.formula}
\textrm{G}_C\textrm{-dim}_R(\textrm{Tr}_C (N)) \, = \,
\textrm{r.grade}_R(\textrm{Tr}_C (N), C).
\end{equation}
Thus, by (\ref{eq.isomorfismos.EXTs.TrC.TrC}),
(\ref{eq2.isomorfismos.EXTs.TrC.TrC}) and
(\ref{eq2.GC.dimensao.formula}), $M$ is reduced
$\textrm{G}_C$-perfect of $\textrm{G}_C$-dimension $n$. Similarly,
\begin{eqnarray*}
\textrm{G}_C\textrm{-dim}_R(\textrm{Ext}_{R}^n (M,\,C)) & = & \textrm{G}_C\textrm{-dim}_R(M) + t -1\\
                        & = & \textrm{r.grade}_R(M, C) + t - 1\\
                        & = & \textrm{grade}_R (\textrm{Ext}_{R}^n (M,\,C)),
\end{eqnarray*}
thus proving that $\textrm{Ext}_{R}^n (M, C)$ is $\textrm{G}_C$-perfect of $\textrm{G}_C$-dimension $n+t-1$. \qed
\end{proof}

\begin{corollary}\label{cor0.do.teo.inspirado.no.cor.3.16.paper3}
Let $R$ be a local ring and $M$ be a $\textrm{\emph{G}}_C$-perfect
$R$-module of grade $n > 0$. Then $\textrm{\emph{Tr}}_C (M)$ is a
reduced $\textrm{\emph{G}}_C$-perfect $R$-module of
$\textrm{\emph{G}}_C$-dimension $1$, and the $R$-module
$\textrm{\emph{Ext}}_{R}^1 (\textrm{\emph{Tr}}_C (M),C)$ is
$\textrm{\emph{G}}_C$-perfect of $\textrm{\emph{G}}_C$-dimension
$n$.
\end{corollary}
\begin{proof}
By Example \ref{ex1.GCperfeito.reduzido} and \cite[10., p.\,68]{golod84},
$M$ is reduced $\textrm{G}_C$-perfect of $\textrm{G}_C$-dimension
$n$ and $\textrm{Ext}_{R}^n (M, C)$ is a $\textrm{G}_C$-perfect
$R$-module of $\textrm{G}_C$-dimension $n = n + 1 - 1$. Now the
assertion is clear by Theorem
\ref{teo.inspirado.no.cor.3.16.paper3} with $t = 1$. \qed
\end{proof}

\begin{corollary}\label{cor.do.teo.inspirado.no.cor.3.16.paper3}
Let $n>0$ and $t>2$ be two integers. Let $M$ be a reduced
$\textrm{\emph{G}}_C$-perfect $R$-module of
$\textrm{\emph{G}}_C$-dimension $n$ such that
$\textrm{\emph{Ext}}_{R}^n (M,C)$ is a
$\textrm{\emph{G}}_C$-perfect $R$-module of
$\textrm{\emph{G}}_C$-dimension $n+t-1$. Then, $M^C$ is reduced
$\textrm{\emph{G}}_C$-perfect of $\textrm{\emph{G}}_C$-dimension
$t-2$ and $\textrm{\emph{Ext}}_{R}^{t-2} (M^C,C)$ is
$\textrm{\emph{G}}_C$-perfect of $\textrm{\emph{G}}_C$-dimension
$n+t-1$.
\end{corollary}
\begin{proof}
By Theorem \ref{teo.inspirado.no.cor.3.16.paper3}, $\textrm{Tr}_C
(M)$ is reduced $\textrm{G}_C$-perfect of $\textrm{G}_C$-dimension
$t$ and moreover $\textrm{Ext}_{R}^t (\textrm{Tr}_C (M),C)$ is
$\textrm{G}_C$-perfect of $\textrm{G}_C$-dimension $n+t-1$.
We have an exact sequence
\begin{equation}\label{eq1.cor.do.teo.inspirado.no.cor.3.16.paper3}
0 \rightarrow M^C \rightarrow P_0^C \rightarrow P_1^C
\rightarrow  \textrm{Tr}_C (M) \rightarrow 0,
\end{equation}
where $P_0$ and $P_1$ are finite projective $R$-modules. Breaking this sequence into two short exact sequences, and using Proposition \ref{prop.lema.1.9.gerko}, we obtain
\begin{equation}\label{eq2.cor.do.teo.inspirado.no.cor.3.16.paper3}
\textrm{G}_C\textrm{-dim}_R(M^C) \, = \, \textrm{G}_C\textrm{-dim}_R(\textrm{Tr}_C (M)) - 2 \, = \, t-2.
\end{equation}
Now, applying Lemma \ref{lema.formula.do.rgrade.geral} to
(\ref{eq1.cor.do.teo.inspirado.no.cor.3.16.paper3}), we get
\begin{equation}\label{eq3.cor.do.teo.inspirado.no.cor.3.16.paper3}
\textrm{r.grade}_R(M^C, C) \, = \, \textrm{r.grade}_R(\textrm{Tr}_C (M), C) - 2 \, = \, t-2.
\end{equation}
By (\ref{eq2.cor.do.teo.inspirado.no.cor.3.16.paper3}) and
(\ref{eq3.cor.do.teo.inspirado.no.cor.3.16.paper3}), $M^C$ is
reduced $\textrm{G}_C$-perfect of $\textrm{G}_C$-dimension $t-2$.
Finally, in virtue of the isomorphism
$\textrm{Ext}_{R}^{t-2} (M^C, C)  \cong \textrm{Ext}_{R}^{t} (\textrm{Tr}_C(M), C)$, we conclude that $\textrm{Ext}_{R}^{t-2} (M^C,C)$ is
$\textrm{G}_C$-perfect of $\textrm{G}_C$-dimension $n+t-1$. \qed
\end{proof}

\medskip

Yet another consequence of Theorem
\ref{teo.inspirado.no.cor.3.16.paper3} is the following generalization of
\cite[Corollary 3.6]{link2013}.

\begin{corollary}\label{cor.3.16.paper3.generalizado}
Let $M$ be a horizontally linked $R$-module. Assume that
$\textrm{\emph{Tor}}^{R}_1 (\textrm{\emph{Tr}}(M), C) = 0$. Let $n$ and $t$ be two positive integers. Then the following statements are equivalent:
\begin{itemize}
  \item[(i)] $M$ is reduced $\textrm{\emph{G}}_C$-perfect of $\textrm{\emph{G}}_C$-dimension $n$ and $\textrm{\emph{Ext}}_{R}^n
(M, C)$ is $\textrm{\emph{G}}_C$-perfect of
$\textrm{\emph{G}}_C$-dimension $n+t$;
  \item[(ii)] $\lambda M \otimes_R C$ is reduced $\textrm{\emph{G}}_C$-perfect
of $\textrm{\emph{G}}_C$-dimension $t$ and
$\textrm{\emph{Ext}}_{R}^t (\lambda M \otimes_R C, C)$ is
$\textrm{\emph{G}}_C$-perfect of $\textrm{\emph{G}}_C$-dimension
$n+t$.
\end{itemize}
\end{corollary}
\begin{proof}
By Theorem \ref{teo.inspirado.no.cor.3.16.paper3}, it
suffices to prove that item (ii) is equivalent to the following assertion:
\begin{itemize}
  \item[(iii)] $\textrm{Tr}_C (M)$ is reduced $\textrm{G}_C$-perfect of $\textrm{G}_C$-dimension $t+1$ and $\textrm{Ext}_{R}^{t+1}
(\textrm{Tr}_C (M), C)$ is $\textrm{G}_C$-perfect of
$\textrm{G}_C$-dimension $n+t$.
\end{itemize}

As seen in Remark \ref{obs.2.11.sadeghi}, there is an exact sequence
\begin{equation}\label{eq3.seg.exata.lambidaM.e.TrC}
0 \rightarrow \lambda M \otimes_R C \rightarrow P \otimes_R C
\rightarrow \textrm{Tr}_C (M) \rightarrow 0,
\end{equation}
where $P$ is a finite projective $R$-module. By \cite[Remark 2.15(i)]{link2016} we have $\textrm{Ext}_{R}^{1}
(\textrm{Tr}_C (M), C) = 0$, because $M$ is  horizontally
linked. Also, recall that
$\textrm{G}_C\textrm{-dim}_R(P \otimes_R C) = 0$ (see the proof of Lemma \ref{obs.syzygy}(ii)). Thus, combining (\ref{eq3.seg.exata.lambidaM.e.TrC}) with Proposition
\ref{prop.item6.golod} and Proposition \ref{prop.item2.golod}, we get the equivalence
$$0<\textrm{G}_C\textrm{-dim}_R(\lambda M \otimes_R C) < \infty
\, \Leftrightarrow \, 0<\textrm{G}_C\textrm{-dim}_R(\textrm{Tr}_C
(M)) < \infty.$$ In this case, by Proposition
\ref{prop.lema.1.9.gerko}, 
\begin{equation}\label{eq.Gcdim123}
\textrm{G}_C\textrm{-dim}_R(\lambda M \otimes_R C)
\, = \, \textrm{G}_C\textrm{-dim}_R(\textrm{Tr}_C (M)) - 1.
\end{equation}
Since $\textrm{Ext}_{R}^{1}
(\textrm{Tr}_C (M), C) = 0$, we have 
$\textrm{r.grade}_R(\textrm{Tr}_C (M), C)>1$ and then, by Lemma \ref{lema.formula.do.rgrade.geral}, \begin{equation}\label{eq2.Gcdim123}
\textrm{r.grade}_R(\lambda M \otimes_R C, C)
\, = \, \textrm{r.grade}_R(\textrm{Tr}_C (M), C) - 1.
\end{equation}
Now it is clear, by (\ref{eq.Gcdim123}) and (\ref{eq2.Gcdim123}), that $\textrm{Tr}_C (M)$ is reduced $\textrm{G}_C$-perfect of
$\textrm{G}_C$-dimension $t+1$ if and only if $\lambda M \otimes_R
C$ is reduced $\textrm{G}_C$-perfect of $\textrm{G}_C$-dimension
$t$. Finally, since $t >0$, the sequence
(\ref{eq3.seg.exata.lambidaM.e.TrC}) gives, in particular,
$\textrm{Ext}_{R}^{t} (\lambda M \otimes_R C, C)  \cong 
\textrm{Ext}_{R}^{t+1} (\textrm{Tr}_C (M), C)$. The desired equivalence (ii)$\Leftrightarrow$(iii) follows. \qed
\end{proof}

\section{Horizontal linkage and (reduced) $\textrm{G}_C$-perfection}\label{sec33}

Besides Corollary \ref{cor.3.16.paper3.generalizado}, much more can be said about the connection between (reduced) $\textrm{G}_C$-perfection and horizontal linkage. This is our goal in this section. In particular, we will generalize some
results of \cite{link2004} and
\cite{link2013}, the latter concerning the problem of characterizing horizontally linked modules.

\begin{lemma}\label{lema.condição.do.grade}
If $M$ is a non-zero horizontally linked $R$-module, then $\textrm{\emph{grade}}_R(M) = 0$.
\end{lemma}
\begin{proof} Suppose $\textrm{{grade}}_R(M) >
0$. Then $M^*=0$, hence $M^{**}=0$. By Proposition \ref{prop.prop.2.6.geng}, there is an isomorphism 
$\textrm{Ext}_{R}^1 (\textrm{Tr}(M), R) \cong M$, so that $\textrm{Ext}_{R}^1 (\textrm{Tr}(M), R)\neq 0$, which contradicts Theorem \ref{teo2.paper1}.  \qed
\end{proof}

\medskip

This lemma immediately gives the following.

\begin{corollary}\label{cor.condição.do.grade.generalizada}
If $M$ is a non-zero horizontally linked 
$\textrm{\emph{G}}_C$-perfect $R$-module, then $\gcdimit (M)=0$.
\end{corollary}

In order to produce a reciprocal of Corollary
\ref{cor.condição.do.grade.generalizada} we need the notion of
Auslander class with respect to $C$, which was introduced in  \cite{foxby72}. Notice that taking $C=R$ trivializes the concept in the sense that every $R$-module lies in the Auslander class with respect to $R$.

\begin{definition}\label{def.classe.de.Auslander} \rm
The \emph{Auslander class with respect to} $C$, denoted by $\mathcal{A}_C$, consists of all $R$-modules $M$ satisfying the following conditions:
\begin{itemize}
             \item[(i)] The natural map $\mu \colon M \rightarrow \textrm{Hom}_{R} (C, M \otimes_R
             C)$ is an isomorphism;
             \item[(ii)] $\textrm{Tor}^{R}_i (C, M) = 0 = \textrm{Ext}_{R}^i (C, M \otimes_R C)$ for all $i > 0$.
\end{itemize}
\end{definition}

\begin{example}\label{ex.classe.de.Auslander}\rm \begin{itemize}
  \item[(i)] If any two $R$-modules in a short exact sequence are in $\mathcal{A}_C$, then so is the third
one. Every module of finite flat dimension lies in $\mathcal{A}_C$. We refer to \cite[1.9]{TakWhi2010}.
  
  \item[(ii)] If $M$ is a finite module over a Cohen-Macaulay local ring $R$ with canonical module $\omega_R$, then $M \in \mathcal{A}_{\omega_R}$
  if and only if $\textrm{G-dim}_R(M) < \infty$ (see \cite[Theorem 1]{foxby75}).
\end{itemize}
\end{example}

\begin{remark}\label{obs.2.15.ii.sadeghi} \rm
Let $M$ be a finite $R$-module such that $\textrm{Tr}(M) \in
\mathcal{A}_C$. Then
$$\textrm{Ext}_{R}^i (\textrm{Tr}_C (M),\,C) \, \cong \, \textrm{Ext}_{R}^i(\textrm{Tr}(M),\,R),$$ for all $i\geq 0$ (see \cite[Remark 2.15 (ii)]{link2016}).
\end{remark}

\begin{theorem}\label{prop.teo1.paper1.generalizado}
Let $M$ be a stable $R$-module such that $\gcdimit
(M)=0$ and $\lambda M \in \mathcal{A}_C$. Then $M$ is
horizontally linked, $\lambda M$ is a stable $R$-module and
$\gcdimit
(\lambda M \otimes_R C)=0$.
\end{theorem}
\begin{proof}
Since $\textrm{G}_{C}\textrm{-dim}_R(M) = 0$, we get
$\textrm{G}_{C}\textrm{-dim}_R(\textrm{Tr}_C (M)) = 0$ by virtue of Proposition \ref{prop.prop.2.6.geng}. In particular,
$\textrm{Ext}_{R}^1 (\textrm{Tr}_C (M), C) = 0$. There is a short
exact sequence $0 \rightarrow \lambda M \rightarrow P \rightarrow
\textrm{Tr}(M) \rightarrow 0$, where $P$ is a projective
$R$-module. As $\lambda M \in \mathcal{A}_C$, it follows by
Example \ref{ex.classe.de.Auslander}(i) that $\textrm{Tr}(M) \in
\mathcal{A}_C$. Hence, by Remark \ref{obs.2.15.ii.sadeghi},
$\textrm{Ext}_{R}^1 (\textrm{Tr}(M), R) = 0$. By Theorem
\ref{teo2.paper1}, $M$ is horizontally linked and then, using
\cite[Proposition 4]{link2004}, we obtain that $\lambda M$ is a stable $R$-module. Now note that $\textrm{Tor}^{R}_1 (\textrm{Tr}(M), C) = 0$, because $\textrm{Tr}(M) \in \mathcal{A}_C$. Thus, by the exact sequence
$0 \rightarrow \lambda M \otimes_R C \rightarrow P \otimes_R C
\rightarrow \textrm{Tr}_C (M) \rightarrow 0$ seen in Remark
\ref{obs.2.11.sadeghi}, we conclude by Proposition
\ref{prop.item2.golod} that $\lambda M \otimes_R C$ has $\textrm{G}_C$-dimension zero. \qed
\end{proof}

\medskip

We point out that Theorem \ref{prop.teo1.paper1.generalizado} is a generalization of \cite[Theorem
1]{link2004}. Next, we provide a number of consequences, the first of which extends \cite[Corollary
2]{link2004} (which requires the ring to be Gorenstein) to the class of Cohen-Macaulay local rings with canonical module.

\begin{corollary}\label{cor.cor2.paper1.generalizado}
Let $R$ be a Cohen-Macaulay local ring with canonical module
$\omega_R$, and let $M$ be a stable maximal Cohen-Macaulay $R$-module. Assume that $\textrm{\emph{G-dim}}_R(\lambda M) < \infty$. Then $M$ is horizontally linked and $\lambda M$ is again a stable maximal Cohen-Macaulay $R$-module.
\end{corollary}
\begin{proof}
By Example \ref{ex.classe.de.Auslander}(ii), $\lambda M \in
\mathcal{A}_{\omega_R}$. By \cite[Lemma 2.16(i)]{link2016}, we get
$\textrm{depth}_R(\lambda M) = \textrm{depth}_R(\lambda M
\otimes_R \omega_R)$. Since $M$ is stable, it follows by \cite[Theorem 32.13 and Corollary
32.14]{fuller} that $\lambda M \neq 0$. Thus, by Proposition \ref{prop.prop.1.3.gerko} and
Theorem \ref{teo.propriedades.Gcdimensao}, we have
$$\textrm{G}_{\omega_R}\textrm{-dim}_R(\lambda M) \, = \,
\textrm{G}_{\omega_R}\textrm{-dim}_R(\lambda M \otimes_R
\omega_R).$$ Also note that $\textrm{G}_{\omega_R}\textrm{-dim}_R(M) =0$. Now the assertion is clear by Theorem \ref{prop.teo1.paper1.generalizado}. \qed
\end{proof}

\begin{corollary}\label{cor.prop.reciproca.do.cor.condicao.do.grade} Let $M$ be a $\textrm{\emph{G}}_C$-perfect $R$-module such that
$\lambda M \in \mathcal{A}_C$. Then, $M$ is horizontally linked if
and only if $M$ is stable with  $\gcdimit (M)=0$.
\end{corollary}
\begin{proof}
If $M$ is horizontally linked, then
$\textrm{G}_C\textrm{-dim}_R(M)=0$ and $M$ is stable, by
Corollary \ref{cor.condição.do.grade.generalizada} and
\cite[Proposition 3]{link2004}, respectively. The converse is given by Theorem \ref{prop.teo1.paper1.generalizado}. \qed
\end{proof}

\begin{corollary}\label{cor3.caracterizacao.LH.TR.generalizado}
Let $R$ be a Cohen-Macaulay local ring with canonical module
$\omega_R$, and let $M$ be a non-zero Cohen-Macaulay $R$-module such that $\textrm{\emph{G-dim}}_R(\lambda M) < \infty $. Then, $M$ is horizontally linked if and only if $M$ is stable maximal Cohen-Macaulay.
\end{corollary}
\begin{proof}
By Proposition \ref{prop.prop.1.3.gerko}, $M$ has finite
$\textrm{G}_{\omega_R}$-dimension. As $\textrm{depth}(R) =
\dim (R)$ and $\textrm{depth}_R(M) = \textrm{dim}_R(M)$,
it follows, by Theorem \ref{teo.propriedades.Gcdimensao} and
\cite[Corollary 2.1.4]{CMr}, that
$$\textrm{G}_{\omega_R}\textrm{-dim}_R(M) \, = \, \dim (R) - \textrm{dim}_R(M) \, = \,
\textrm{ht}(\textrm{ann}_R(M)) \, = \, \textrm{grade}_R(M),$$ where $\textrm{ht}(\textrm{ann}_R(M))$ stands for the height of the annihilator $\textrm{ann}_R(M)$ of $M$ -- this coincides with $\textrm{grade}_R(M)$ because $R$ is Cohen-Macaulay.
It follows that $M$ is $\textrm{G}_{\omega_R}$-perfect. Note that
$\textrm{G}_{\omega_R}\textrm{-dim}_R(M) = 0$ if and only if
$M$ is maximal Cohen-Macaulay. Since $\textrm{G-dim}_R(\lambda
M) < \infty$, we have $\lambda M \in \mathcal{A}_{\omega_R}$ by
Example \ref{ex.classe.de.Auslander}(ii). Now Corollary
\ref{cor.prop.reciproca.do.cor.condicao.do.grade} yields the result. \qed
\end{proof}

\medskip

In the Gorenstein case (so that every finite module has finite G-dimension) we immediately get the following.

\begin{corollary}\label{cor3.caracterizacao.LH.TR.generalizado}
Let $R$ be a Gorenstein local ring and let $M$ be a non-zero Cohen-Macaulay $R$-module. Then, $M$ is horizontally linked if and only if $M$ is stable maximal Cohen-Macaulay.
\end{corollary}

The next three results are of self-interest and will furthermore be useful in the proof of Theorem \ref{teo.prop.3.5ii.paper3.generalizada}.

\begin{proposition}\label{lema.condição.suficiente.para.EXT.TrC.nulo} Let $M$ be a reduced $\textrm{\emph{G}}_C$-perfect $R$-module of $\textrm{\emph{G}}_C$-dimension $n$. Assume that
$\textrm{\emph{grade}}_R(\textrm{\emph{Ext}}_R^n(M,
C))\geq n+1$. Then\, $\textrm{\emph{Ext}}_{R}^1
(\textrm{\emph{Tr}}_C (M), C) = 0$.
\end{proposition}
\begin{proof} By the hypothesis on the grade, we have $\textrm{Ext}_R^n(\textrm{Ext}_R^n(M, C), C) = 0$. On the other hand, applying Proposition
\ref{prop.EXT.TrC} (with $i=1$), we get an isomorphism $\textrm{Ext}_{R}^{n}(\textrm{Ext}_{R}^n (M, C), C) \cong
\textrm{Ext}_{R}^{1} (\textrm{Tr}_C (M), C)$. The result follows. \qed
\end{proof}

\medskip

In Proposition \ref{prop.4.16.stablemodule.generalizada} we showed that if $M$ is a finite $R$-module of finite $\textrm{\emph{G}}_C$-dimension, then
$\textrm{{grade}}_R(\textrm{{Ext}}_{R}^i(M, C))  \geq  i$ for all $i > 0$. This can be improved if moreover $M$ is horizontally linked, as follows.

\begin{proposition} \label{prop.LH.Gcdim.finita.implica.gradeEXT.maior.n}
Let $M$ be a horizontally linked $R$-module of finite
$\textrm{\emph{G}}_C$-dimension. Then, for any given $i > 0$,
$$\textrm{\emph{grade}}_R(\textrm{\emph{Ext}}_{R}^i (M,\,C)) \, \geq \, i + 1.$$ 
\end{proposition}
\begin{proof}
By \cite[Remark 2.15(i)]{link2016}, $\textrm{Ext}_{R}^1
(\textrm{Tr}_C (M), C) = 0$. By Lemma \ref{obs.syzygy}(ii), $M$ is a $C$-syzygy module, i.e there exists an exact sequence $0
\rightarrow M \rightarrow L$, where $L = P \otimes_R C$ and $P$ is
a finite projective $R$-module. Recall that 
 $\textrm{G}_C\textrm{-dim}_R(L)=0$.
Set $N = \textrm{Coker}\, (M \rightarrow
L)$. Since $\textrm{Ext}_{R}^i (L, C) = 0$ for all $i > 0$, we get $\textrm{Ext}_{R}^i (M, C) \cong \textrm{Ext}_{R}^{i+1} (N, C),
\textrm{ for all } i > 0$.
As $\textrm{G}_C\textrm{-dim}_R(M) < \infty$, we must have $\textrm{G}_C\textrm{-dim}_R(N) < \infty$ by Proposition
\ref{prop.item6.golod}. Therefore, using Proposition
\ref{prop.4.16.stablemodule.generalizada}, we obtain
$\textrm{grade}_R(\textrm{Ext}_{R}^i (M, C))  =  \textrm{grade}_R(\textrm{Ext}_{R}^{i+1} (N, C)) \geq  i + 1$, for all $i > 0$. \qed
\end{proof}

\begin{proposition}\label{cor.LH.Gperf.redu.formula.do.rgrade.e.gradeC} Let $M$ be a horizontally linked reduced
$\textrm{\emph{G}}_C$-perfect $R$-module of
$\textrm{\emph{G}}_C$-dimension $n$. Assume that
$\textrm{\emph{Tor}}^{R}_1 (\textrm{\emph{Tr}}(M), C) = 0$. Then,
$$\textrm{\emph{r.grade}}_R(M, C) + \textrm{\emph{r.grade}}_R (\lambda M \otimes_R C, C) \, = \, \textrm{\emph{grade}}_R (\textrm{\emph{Ext}}_{R}^n (M,\,C)).$$
\end{proposition}
\begin{proof} Consider a short exact sequence $0 \rightarrow \lambda M \otimes_R C \rightarrow P \otimes_R C\rightarrow \textrm{Tr}_C (M) \rightarrow 0$, where $P$ is a projective $R$-module
(see Remark \ref{obs.2.11.sadeghi}). Since $M$ is horizontally linked, \cite[Remark 2.15(i)]{link2016} gives
$\textrm{Ext}_{R}^{1} (\textrm{Tr}_C (M), C)=0$, which forces
$\textrm{r.grade}_R(\textrm{Tr}_C (M), C) > 1$. By Lemma
\ref{lema.formula.do.rgrade.geral}, we have
$\textrm{r.grade}_R (\lambda M \otimes_R C, C) =
\textrm{r.grade}_R (\textrm{Tr}_C (M), C) - 1 $. Now the
assertion is clear by Proposition
\ref{prop.Gperf.redu.formula.do.rgrade.e.gradeC}. \qed
\end{proof}

\medskip

We are now prepared to prove the following theorem, which characterizes, numerically, when a reduced $\textrm{G}_C$-perfect module (subject to a few additional conditions) is horizontally linked. It generalizes and improves \cite[Proposition
3.5(ii)]{link2013}.

\begin{theorem}\label{teo.prop.3.5ii.paper3.generalizada}
Let $M$ be a stable reduced $\textrm{\emph{G}}_C$-perfect
$R$-module of $\textrm{\emph{G}}_C$-dimension $n$. Assume that
$\lambda M \in \mathcal{A}_C$. Then the following assertions are
equivalent:
\begin{itemize}
  \item[(i)] $M$ is horizontally linked;
  \item[(ii)] $\textrm{\emph{r.grade}}_R(M, C) + \textrm{\emph{r.grade}}_R(\lambda M) = \textrm{\emph{grade}}_R (\textrm{\emph{Ext}}_R^n(M, C))$;
  \item[(iii)] $\textrm{\emph{grade}}_R(\textrm{\emph{Ext}}_{R}^n (M, C)) \geq n+1$;
  \item[(iv)] $\textrm{\emph{grade}}_R (\textrm{\emph{Ext}}_{R}^i(M, C)) \geq i+1,$ for all $i > 0$.
\end{itemize}
\end{theorem}
\begin{proof} Let us first prove the implication (i) $\Rightarrow$ (ii). Since $\lambda M \in \mathcal{A}_C$, it
follows by \cite[Theorem 4.1 e Corollary 4.2]{TakWhi2010} that
$\textrm{Ext}_R^i(\lambda M \otimes_R C, C)  \cong 
\textrm{Ext}_R^i(\lambda M, R)$, for all $i > 0$. Then
$\textrm{r.grade}_R(\lambda M) = \textrm{r.grade}_R
(\lambda M \otimes_R C, C)$. Consider an exact sequence
$0 \rightarrow \lambda M \rightarrow P \rightarrow \textrm{Tr}(M)
\rightarrow 0$, where $P$ is a finite projective $R$-module. By
Example \ref{ex.classe.de.Auslander}(i), $\textrm{Tr}(M) \in
\mathcal{A}_C$. In particular, $\textrm{Tor}^{R}_1 (\textrm{Tr}(M), C) = 0$. Therefore, by Proposition
\ref{cor.LH.Gperf.redu.formula.do.rgrade.e.gradeC}, we have
$$\textrm{r.grade}_R(M, C) + \textrm{r.grade}_R(\lambda M) \, = \, \textrm{grade}_R(\textrm{Ext}_R^n(M,\,C)).$$

To show that (ii) $\Rightarrow$ (iii), simply note that $\textrm{r.grade}_R(M, C)=n$ and
$\textrm{r.grade}_R(\lambda M) \geq 1$, and hence
$$\textrm{grade}_R(\textrm{Ext}_R^n(M,\,C)) \, = \,
\textrm{r.grade}_R(M, C) + \textrm{r.grade}_R(\lambda M) \, \geq \, n+1.$$

Now we prove (iii) $\Rightarrow$ (i). By Proposition
\ref{lema.condição.suficiente.para.EXT.TrC.nulo}, we get
$\textrm{Ext}_{R}^{1} (\textrm{Tr}_C (M), C) = 0$. As $\lambda M
\in \mathcal{A}_C$, we also have $\textrm{Tr}(M) \in
\mathcal{A}_C$. Thus, by Remark \ref{obs.2.15.ii.sadeghi} it follows that
$\textrm{Ext}_{R}^{1} (\textrm{Tr}(M), R) = 0$. Because in addition $M$ is stable, Theorem \ref{teo2.paper1} gives that $M$ is horizontally linked.

Thus, we have shown that the statements (i), (ii), (iii) are equivalent. To conclude, the implication (iv) $\Rightarrow$ (iii) is obvious, and notice for instance that (i) $\Rightarrow$ (iv)
follows readily by Proposition
\ref{prop.LH.Gcdim.finita.implica.gradeEXT.maior.n}. \qed
\end{proof}

\medskip

We close the section with the expected generalization of \cite[Proposition
3.5(i)]{link2013}.

\begin{proposition}\label{cor.prop.3.5.paper3.generalizada}
Let $M$ be a reduced $\textrm{\emph{G}}_C$-perfect $R$-module of
$\textrm{\emph{G}}_C$-dimension $n$. Assume that $\lambda M \in
\mathcal{A}_C$. Then $$\textrm{\emph{Ext}}_{R}^i (\lambda M,\,R) \, \cong \,
\textrm{\emph{Ext}}_{R}^{i+n} (\textrm{\emph{Ext}}_{R}^n (M,\,C),\,C)$$ for all $i>0$.
\end{proposition}
\begin{proof} As seen in the proof of Theorem \ref{teo.prop.3.5ii.paper3.generalizada}, we have $\textrm{Ext}_{R}^i (\lambda M \otimes_R C, C) \cong
\textrm{Ext}_{R}^i (\lambda M, R)$ for all $i>0$, and moreover $\textrm{Tor}^{R}_1 (\textrm{Tr}(M),
C) = 0$. Now the result follows by Corollary \ref{cor.EXT.Lambda}.
\qed
\end{proof}

\section{Reduced $\textrm{G}_C$-perfection versus $C$-$k$-torsionlessness} \label{sec34}

We begin this last section by recalling the following definition, given in \cite[Definition
2]{huang2001}, which generalizes the notion of $k$-torsionless module, where $k\geq 0$ is an integer.

\begin{definition}\label{def2.huang} \rm
Let $M$ be a finite $R$-module and let $k\geq0$ be an integer. Then $M$ is called $C$-$k$-\emph{torsionless} if either $k=0$ or $k\geq 1$ and $$\textrm{Ext}_{R}^i(\textrm{Tr}_C(M),\,C) \, = \, 0$$ for all $1 \leq i \leq k$. 
\end{definition}

It is worth mentioning that \cite{huang2001} employed the terminology {\it $\omega$-$k$-torsionfree}, and that yet other names have appeared in the literature; see, for instance, \cite{arayaiima} and \cite{link2015}. 

In general, the classes of $C$-$k$-torsionless modules and reduced
$\textrm{G}_C$-perfect modules are incomparable in the sense that there is no inclusion between them, as the following remark
illustrates.

\begin{remark} \label{ex0.GCperfeito.reduzido.e.cktorsionless} \rm
Let $k$ be a positive integer and let $M$ be a non-zero finite $R$-module.
\begin{itemize}
  \item[(i)] If $\textrm{G}_{C}\textrm{-dim}_R(M) = 0$ then by Proposition \ref{prop.prop.2.6.geng}(iii) and Remark \ref{obs.carac.GCdim.zero} it follows that $M$ is $C$-$k$-torsionless, but obviously $M$ cannot be reduced $\textrm{G}_C$-perfect.
  \item[(ii)] If $M$ is $\textrm{G}_C$-perfect of positive grade, then $M$ is reduced
  $\textrm{G}_C$-perfect, and applying Example \ref{ex2.GCperfeito.reduzido}(i) (with $n=1$) yields in particular that the module $\mathcal{T}^C_{1} (M) = \textrm{Tr}_C (M)$ has reduced grade 1, thus forcing $\textrm{Ext}_{R}^1 (\textrm{Tr}_C (M), C) \neq
  0$, which prevents $M$ from being $C$-$k$-torsionless.
\end{itemize}
\end{remark}

The next result provides a necessary and sufficient condition for
a reduced $\textrm{G}_C$-perfect module to be $C$-$k$-torsionless.

\begin{proposition}\label{prop.cktorsionless.gcperfeito}
Let $M$ be a reduced $\textrm{\emph{G}}_C$-perfect $R$-module of
$\textrm{\emph{G}}_C$-dimension $n$, and let $k\geq0$ be an
integer. Then, $M$ is $C$-$k$-torsionless if and only if
$$\textrm{\emph{grade}}_R(\textrm{\emph{Ext}}_{R}^n (M,\,C)) \,
\geq \, n + k.$$
\end{proposition}
\begin{proof}
Assume that $M$ is $C$-$k$-torsionless. By definition,
$\textrm{Ext}_{R}^i(\textrm{Tr}_C (M), C) = 0$ for all $1 \leq i
\leq k$, and so $\textrm{r.grade}_R(\textrm{Tr}_C (M), C) \geq
k+1$. By Proposition
\ref{prop.Gperf.redu.formula.do.rgrade.e.gradeC}, we have
$$\textrm{grade}_R(\textrm{Ext}_{R}^n (M,\,C)) \, = \,  \textrm{r.grade}_R(\textrm{Tr}_C (M), C) + n - 1 \, \geq \, n+k.$$
Conversely, suppose $\textrm{grade}_R(\textrm{Ext}_{R}^n
(M, C)) \geq n+k.$ By using Proposition
\ref{prop.Gperf.redu.formula.do.rgrade.e.gradeC} again, we get
$$\textrm{r.grade}_R(\textrm{Tr}_C (M), C) \, = \, \textrm{grade}_R(\textrm{Ext}_{R}^n (M,\,C)) - n + 1 \, \geq \, k+1.$$ Therefore, $\textrm{Ext}_{R}^i(\textrm{Tr}_C (M), C) = 0$ for all $1 \leq i \leq k$, i.e. $M$ is $C$-$k$-torsionless. \qed
\end{proof}

\medskip

In this context we may naturally raise the following question.

\begin{question}\rm
Let $M$ be a $C$-$k$-torsionless $R$-module. Under what conditions
is $M$ reduced $\textrm{G}_C$-perfect?
\end{question}

\begin{corollary}\label{cor.prop.cktorsionless.gcperfeito}
Let $M$ be a stable reduced $\textrm{\emph{G}}_C$-perfect
$R$-module of $\textrm{\emph{G}}_C$-dimension $n$. Assume that
$\lambda M \in \mathcal{A}_C$. The following conditions are
equivalent:
\begin{itemize}
  \item[(i)] $M$ is horizontally linked;
  \item[(ii)] $M$ is $C$-$1$-torsionless;
  \item[(iii)] $M$ is $C$-syzygy.
\end{itemize}
\end{corollary}
\begin{proof} First, by Proposition
\ref{prop.cktorsionless.gcperfeito}, the assertion (ii) is equivalent to $\textrm{{grade}}_R(\textrm{{Ext}}_{R}^n (M, C)) \geq n+1$, so that the equivalence between (i) and (ii) is given by Theorem \ref{teo.prop.3.5ii.paper3.generalizada}. Now note that ${\rm Tr}(M)\in \mathcal{A}_C$, and hence, by Lemma \ref{obs.syzygy} and Remark \ref{obs.2.15.ii.sadeghi}, the assertion (iii) is seen to be equivalent to $\textrm{Ext}_{R}^1(\textrm{Tr}(M), R) = 0$. Since $M$ is stable, this vanishing is equivalent to (i) by virtue of Theorem \ref{teo2.paper1}. 
\qed
\end{proof}

\medskip

In the example below (which appears to be new), we are concerned with the description of a family of modules lying in the
intersection of the classes of
$C$-$k$-torsionless modules and reduced $\textrm{G}_C$-perfect modules.

\begin{example} \label{ex.GCperfeito.reduzido.e.cktorsionless} \rm
Let $n$ be a positive integer, $k$ a non-negative integer, and $M$
a non-zero finite $R$-module such that $\textrm{grade}_R(M) \geq
n+k$. Then $\mathcal{T}^C_{n} (M)$ is both $C$-$k$-torsionless and
reduced $\textrm{G}_C$-perfect of $\textrm{G}_C$-dimension $n$.
\end{example}
\begin{proof}
Set $M' = \mathcal{T}^C_{n} (M)$. By Example
\ref{ex2.GCperfeito.reduzido}(i), $M'$ is a reduced
$\textrm{G}_C$-perfect $R$-module of $\textrm{G}_C$-dimension $n$. Thus, 
by Proposition
\ref{prop.cktorsionless.gcperfeito}, proving that $\textrm{grade}_R(\textrm{Ext}_{R}^n (M',\,C)) \,
\geq \, n+k$ will do the job. By
Proposition \ref{prop.Gperf.redu.formula.do.rgrade.e.gradeC}, we
have
\begin{equation}\label{eq1.ex.GCperfeito.reduzido.e.cktorsionless}
\textrm{grade}_R(\textrm{Ext}_{R}^n (M',\,C)) \, = \, n +
\textrm{r.grade}_R(\textrm{Tr}_C (M'), C) - 1.
\end{equation}
By Proposition \ref{prop.lemma.2.12.paper5}, we have an exact sequence $0 \rightarrow \Omega^{n-1} M
\rightarrow \textrm{Tr}_C (M') \rightarrow X \rightarrow 0$, where
$\textrm{G}_C\textrm{-dim}_R(X)=0$. Thus,
\begin{equation}\label{eq2.ex.GCperfeito.reduzido.e.cktorsionless}
\textrm{r.grade}_R(\Omega^{n-1} M, C) \, = \, \textrm{r.grade}_R
(\textrm{Tr}_C (M'), C).
\end{equation}
Now consider the exact sequence $$0 \rightarrow \Omega^{n-1} M
\rightarrow P_{n-2} \rightarrow \cdots \rightarrow P_0 \rightarrow
M \rightarrow 0,$$ where $P_j$ is a finite projective $R$-module for each $j = 0, \ldots, n-2$. Because $\textrm{grade}_R(M) \geq n+k \geq n > n-1$ and $\textrm{Ext}_{R}^i (P_j, C)=0$ for all $i > 0$ and $j = 0, \ldots, n-2$, it follows by Lemma \ref{lema.formula.do.rgrade.geral} that
\begin{equation}\label{eq3.ex.GCperfeito.reduzido.e.cktorsionless}
\textrm{r.grade}_R(\Omega^{n-1} M, C) \, = \, \textrm{r.grade}_R
(M, C) - n + 1.
\end{equation}
Therefore, by (\ref{eq1.ex.GCperfeito.reduzido.e.cktorsionless}),
(\ref{eq2.ex.GCperfeito.reduzido.e.cktorsionless}), and
(\ref{eq3.ex.GCperfeito.reduzido.e.cktorsionless}), 
$$\textrm{grade}_R(\textrm{Ext}_{R}^n (M',\,C)) \, = \,
\textrm{r.grade}_R(M, C) \, = \, \textrm{grade}_R(M) \, \geq \, n+k,$$ as needed. \qed
\end{proof}

\medskip

Let us close the paper with a related example in the case $C=R$. Let $n$ be a positive integer and let
$M$ be a finite $R$-module with (possibly infinite) projective dimension at least $n$, so that $\Omega^{n-1} M$ is not  projective. Thus, with the
objective of verifying the statement given below,
we may assume that the transpose $\mathcal{T}_{n}(M) = 
\textrm{Tr}\Omega^{n-1} M$ is a stable $R$-module. Finally, we use Corollary \ref{cor.prop.cktorsionless.gcperfeito}
and Example \ref{ex.GCperfeito.reduzido.e.cktorsionless}.

\begin{example} \label{ex2.GCperfeito.reduzido.e.cktorsionless} \rm If, in addition, $\textrm{grade}_R(M) \geq n+k$ for some positive integer $k$, then $\mathcal{T}_{n}(M)$ is $k$-torsionless, reduced $\textrm{G}$-perfect of G-dimension $n$, and horizontally linked.
\end{example}

\bigskip

\noindent{\bf Acknowledgements.} The first author was partially supported by CNPq (Brazil), grant 301029/2019-9. The second author was supported by a CAPES (Brazil) Doctoral Scholarship.

\end{document}